\documentclass[reqno]{amsart}

\usepackage[utf8]{inputenc}
\usepackage[T1]{fontenc}
\usepackage[british]{babel}
\usepackage{amssymb}
\usepackage{amsfonts}
\usepackage{mathabx}
\usepackage{calrsfs}
\usepackage{mathbbol}
\usepackage{comment}
\usepackage{etoolbox}
\usepackage{graphicx}
\usepackage{hyperref}
\usepackage{xcolor}
\usepackage{seqsplit}
\usepackage[all]{xy}
\usepackage{multicol}

\newcommand{\defn}[1]{\textbf{#1}}
\newcommand{\noproof}{\hfill\qed}

\theoremstyle{plain}
\newtheorem{theorem}[subsection]{Theorem}
\newtheorem{proposition}[subsection]{Proposition}
\newtheorem{corollary}[subsection]{Corollary}
\newtheorem{lemma}[subsection]{Lemma}

\theoremstyle{definition}
\newtheorem{definition}[subsection]{Definition}
\newtheorem{notation}[subsection]{Notation}

\theoremstyle{remark}
\newtheorem{remark}[subsection]{Remark}
\newtheorem{example}[subsection]{Example}
\newtheorem{examples}[subsection]{Examples}

\numberwithin{equation}{section}
\allowdisplaybreaks

\newcommand{\kar}{\mathrm{char}}
\newcommand{\cre}{\mathrm{Cr}}

\newcommand{\N}{\mathbb{N}}
\newcommand{\Z}{\mathbb{Z}}
\newcommand{\fK}{\mathbb{K}}

\newcommand{\cosmash}{\diamond}

\newcommand{\Singular}{{\sc Singular}}
\newcommand{\Mathematica}{{\sc Mathematica}}

\newcommand{\ttab}[5]{\begin{tabular}{c} $#3 \qquad #4$ \\\hline\\[-.5ex] $\begin{aligned} #5 \end{aligned}$\end{tabular}\bigskip}

\makeatletter
\renewcommand{\do}[1]{\@namedef{c#1}{\ensuremath{\mathcal{#1}}}}
\makeatother
\docsvlist{A,B,C,D,E,P,U,V} 

\newcommand{\SET}{\mathsf{Set}}
\newcommand{\MAG}{\mathsf{Mag}}
\newcommand{\GP}{\mathsf{Gp}}
\newcommand{\BOOL}{\mathsf{Bool}}
\newcommand{\RNG}{\mathsf{Rng}}
\newcommand{\FUN}{\mathsf{{Fun}}}
\newcommand{\VECT}{\mathsf{Vect_{\fK}}}
\newcommand{\ALG}{\mathsf{Alg_{\fK}}}
\newcommand{\ABALG}{\mathsf{Ab_{\fK}}}
\newcommand{\CA}{\mathsf{CA_{\fK}}}

\newcommand{\CAp}{\mathsf{CA}^p_{\fK}}

\newcommand{\ANTI}{\mathsf{Anti_{\fK}}}
\newcommand{\ALT}{\mathsf{Alt_{\fK}}}
\newcommand{\TRIV}{\mathsf{Triv_{\fK}}}
\newcommand{\LIE}{\mathsf{Lie_{\fK}}}
\newcommand{\NIL}{\mathsf{Nil_2(Alg_{\fK})}}

\newcommand{\LLL}{\alpha}
\newcommand{\LLR}{\beta}
\newcommand{\LRL}{\gamma}
\newcommand{\LRR}{\delta}

\newdir{>>}{{}*!/3.5pt/:(1,-.2)@^{>}*!/3.5pt/:(1,+.2)@_{>}*!/7pt/:(1,-.2)@^{>}*!/7pt/:(1,+.2)@_{>}}
\newdir{ >>}{{}*!/8pt/@{|}*!/3.5pt/:(1,-.2)@^{>}*!/3.5pt/:(1,+.2)@_{>}}
\newdir{ |>}{{}*!/-3.5pt/@{|}*!/-8pt/:(1,-.2)@^{>}*!/-8pt/:(1,+.2)@_{>}}
\newdir{ >}{{}*!/-8pt/@{>}}
\newdir{>}{{}*:(1,-.2)@^{>}*:(1,+.2)@_{>}}
\newdir{<}{{}*:(1,+.2)@^{<}*:(1,-.2)@_{<}}

\def\pullback{
\ar@{-}[]+R+<6pt,-1pt>;[]+RD+<6pt,-6pt>%
\ar@{-}[]+D+<1pt,-6pt>;[]+RD+<6pt,-6pt>}

\title[Associativity and the cosmash product]{Associativity and the cosmash product\\ in operadic varieties of algebras}

\author{Ülo Reimaa}
\author{Tim\ Van~der Linden}
\author{Corentin Vienne}

\email{tim.vanderlinden@uclouvain.be}
\email{ulo.reimaa@ut.ee}
\email{corentin.vienne@uclouvain.be}

\address[Ülo Reimaa, Tim Van~der Linden, Corentin Vienne]{Institut de Recherche en Math\'ematique et Physique, Universit\'e catholique de Louvain, che\-min du cyclotron~2 bte~L7.01.02, B--1348 Louvain-la-Neuve, Belgium}

\address[Ülo Reimaa]{Institute of Mathematics and Statistics, Faculty of Science and Technology, University of Tartu, Narva mnt 18, 51009 Tartu, Estonia}

\thanks{Research of the first author was supported by the Estonian Research Council grant PUTJD948. The second author is a Senior Research Associate of the Fonds de la Recherche Scientifique--FNRS. Research of the third author was supported by the {Fonds Thelam} of the {Fondation Roi Baudouin}. Computational resources have been provided by the Consortium des Équipements de Calcul Intensif (CÉCI), funded by the Fonds de la Recherche Scientifique de Belgique (F.R.S.-FNRS) under Grant No.~2.5020.11 and by the Walloon Region.}


\begin{document}

\begin{abstract}
    In this article, we characterise the operadic variety of commutative associative algebras over a field via a (categorical) condition: the associativity of the so-called \emph{cosmash product}. This condition, which is closely related to commutator theory, is quite strong: for example, groups do not satisfy it. However, in the case of commutative associative algebras, the cosmash product is nothing more than the tensor product; which explains why in this case it is associative. We prove that in the setting of operadic varieties of algebras over a field, it is the only example. Further examples in the non-operadic case are also discussed.
\end{abstract}

\maketitle

\tableofcontents

\section{Introduction}

The original question we set out to answer at the onset of this work was to characterise, by means of some universal construction, when the objects in a variety of algebras over a field $\fK$ have a multiplication that is associative. Note that what we here call an \emph{algebra} is a $\fK$-vector space $A$ equipped with a bilinear multiplication ${\cdot\colon A\times A\to A}$, which is not necessarily associative or unitary. We let $\ALG$ denote the category of such $\fK$-algebras, where morphisms are linear maps that preserve the multiplication. A~\emph{variety of $\fK$-algebras} is an equational class of algebras over~$\fK$: any subvariety of $\ALG$, as for instance the variety of associative algebras, which satisfy the identity $x(yz)=(xy)z$, or the variety of Lie algebras, which satisfy the Jacobi identity ($x(yz)+y(zx)+z(xy)=0$) and anti-commutativity ($xy=-yx$).

In previous work, characterisations of the variety $\LIE$ of Lie algebras over an infinite field $\fK$ (such that $\kar(\fK)\neq 2$) as a subvariety of $\ALG$ were obtained, in essentially two different ways: in~\cite{GM-VdL2,GM-VdL3}, it is shown that $\LIE$ is the only non-abelian \emph{locally algebraically cartesian closed}~\cite{Gray2012} subvariety of $\ALG$; and in~\cite{Edinburgh}, the variety $\LIE$ is shown to be the only subvariety of $\ALG$ whose actions are representable~\cite{BJK2}. So we have two independent categorical descriptions of the variety $\LIE$, and thus of the Jacobi identity (however, only for anti-commutative algebras).

This naturally led to our present question: \emph{How to characterise \textbf{associative} algebras?} The answer we found---and this is the subject of our present article---is a categorical description of the variety of \emph{commutative} associative algebras over~$\fK$ amongst all so-called \emph{operadic} varieties over~$\fK$, where a set of multilinear identities suffices to describe the algebras in the variety. As it turns out, the variety of commutative associative algebras over $\fK$ is the only operadic variety whose so-called \emph{cosmash products} are not just symmetric, but also associative in the sense of Carboni and Janelidze---see Theorem~\ref{main-result}. Thus, associativity of an algebra (a~``microcosmic'' property at object level) is characterised by associativity on the categorical level (a ``macrocosmic'' property). This is an instance of what seems to be a common phenomenon, comparable, for instance, with the definition of monoids in a monoidal category (which, however, fits a slightly different pattern in that there, a correspondence between two types of \emph{structure} is apparent).

\subsection{The cosmash product}
In~\cite{Smash}, A.\ Carboni and G.\ Janelidze extend the definition of the classical \emph{smash product} from pointed topological spaces to pointed objects in  suitable categories. On the way, they study a condition they call \emph{smash associativity}, which is basically the requirement that any ternary smash product $X\wedge Y\wedge Z$ is canonically isomorphic to either one of the repeated binary smash products $(X\wedge Y)\wedge Z$ and $X\wedge (Y\wedge Z)$. Depending on the surrounding category, this may or may not happen.

In the present article, working in the context of algebras over a field, we interest ourselves in the dual notion: \emph{cosmash associativity}. Indeed, binary \emph{co}smash products are well-known to be useful in the development of \emph{commutator theory} in general algebraic contexts~\cite{MM-NC,HVdL,Actions}, essentially because they may be seen as ``objects of formal commutators''. For example, in the case of groups, the binary cosmash product $X\diamond Y$ of two groups $X$ and $Y$ is generated by commutator words of the form $xyx^{-1}y^{-1}$ (see Section~\ref{Section Definition Cosmash}). The binary definition of the cosmash product is the $n=2$ case of a more general $n$-ary definition~\cite{Smash, HVdL}.

In general, iterating the binary cosmash product does not give the desired higher-dimensional notion of $n$-fold cosmash product. For example, letting $X$, $Y$, $Z$ be three objects in a (suitable) category, we write $X\diamond Y \diamond Z$ for the ternary cosmash product ($n=3$) of $X$, $Y$ and $Z$. This object need not be isomorphic to either $(X\diamond Y)\diamond Z$ or $X \diamond(Y\diamond Z)$---in Section~\ref{Section Cosmash Associativity}, we explain why this isomorphism almost never holds for groups or Lie algebras. \emph{Cosmash associativity} is the condition that these three objects are canonically isomorphic. The goal of this paper is to investigate this categorical condition in the context of varieties of algebras over a field. Our hunch was that it must be a relatively strong condition, since even Lie algebras do not satisfy it. In fact, our main results, which are Theorem~\ref{deg-2-identities} and Theorem~\ref{main-result}, say that this condition is so strong that, in the context of homogeneous varieties of non-associative algebras, it characterises when the multiplication is both commutative and associative.

\subsection{Structure of the text}
The paper is organised as follows. Section~\ref{Section Varieties} gives a short survey of the mathematical context: non-associative algebras and their varieties. Special attention goes to \emph{homogeneous} and \emph{operadic} varieties of algebras over a field, since our main theorems appear in those precise settings. Section~\ref{Section Definition Cosmash} defines cosmash products in their binary, ternary and general form, and gives an overview of some basic properties and examples. Section~\ref{Higgins Commutator} recalls how cosmash products are related to (Higgins) commutators. Section~\ref{Section Cosmash Associativity} discusses the definition of cosmash associativity: certain canonical comparison morphisms are isomorphisms. We study examples and have a first look at some computational techniques we use later on. Section~\ref{Section Cosmash Is Tensor} recalls that in the variety of commutative associative algebras, the cosmash product is the tensor product; it follows that this variety is cosmash associative.

In order to properly understand the statements of our two main theorems, reading Sections~\ref{Section Varieties},~\ref{Section Definition Cosmash} and~\ref{Section Cosmash Associativity} should be sufficient. The details of their proofs and the techniques involved are explained in Sections~\ref{Section Surjectivity},~\ref{Section Injectivity} and~\ref{Section Operadic Case}.

Section~\ref{Section Surjectivity} looks at the consequences of \emph{surjectivity} of the comparison maps. Proposition~\ref{Proposition Lambda Rules} shows that it is related to the validity of certain degree-three equations, such as for instance the equation $(yx)z = \LLR_1 x(yz) + \LLR_2 x(zy) + \LLR_3 (yz)x + \LLR_4 (zy)x$ which allows us to ``pull the $x$ out of the parentheses''. Section~\ref{Section Injectivity} is devoted to the \emph{injectivity} of the comparison maps. First we give examples of varieties where these are not injective. We then use computer algebra to prove in Proposition~\ref{Proposition Degree Two} that a homogeneous variety which is cosmash associative should satisfy an identity of degree less than or equal to $2$. From this we may deduce our first main result: Theorem~\ref{deg-2-identities}, which says that if a homogeneous variety of $\fK$-algebras is cosmash associative, then it is a subvariety of the variety of commutative associative $\fK$-algebras. Section~\ref{Section Operadic Case} interprets this in the particular setting of operadic varieties (for example, when the characteristic of the field is zero). In this context, we understand with Theorem~\ref{main-result} that in a non-abelian cosmash associative operadic variety, no non-trivial identities can be added to commutativity and associativity. In other words, the variety of commutative associative algebras is characterised by cosmash associativity. Section~\ref{Section Non-Operadic Case} is an attempt to extend our results to the case of non-operadic varieties. We study some examples of varieties which could not be considered in the setting of the main theorems. The final Section~\ref{Section Final Remarks} discusses some ideas for further development of the theory.

\section{(Operadic) varieties of algebras over a field}\label{Section Varieties}

In this section, we introduce the algebraic setting where our main results are expressed: the context of \emph{varieties of non-associative $\fK$-algebras}. Those are basically collections of algebras over $\fK$ satisfying a chosen set of equations of a certain type, i.e., an equational class. Unless otherwise stated, $\fK$ will denote a field. Some of our results hold when $\fK$ is an integral domain, or even just a commutative ring with unit.

Let us recall some well-known concepts in order to present things in a consistent way. The interested reader may find a more exhaustive presentation of this material in~\cite{VdL-NAA}. By a \defn{(non-associative) $\fK$-algebra} we mean a $\fK$-vector space equipped with a bilinear operation called the \emph{multiplication}. We do not assume the existence of a unit element, or place any assumptions on the multiplication besides bilinearity. We let $\ALG$ denote the category of not necessarily associative $\fK$-algebras. In this category, the morphisms are the $\fK$-linear maps which preserve the multiplication. In~\cite{Zinbiel}, this type of algebra is called a \defn{magmatic algebra}.

We consider the functor $\SET \rightarrow \ALG$ which associates to a set $X$ the \emph{free $\fK$-algebra} generated by elements of $X$. This functor is left adjoint to the forgetful functor and factorises through the \emph{free magma functor} $M\colon\SET \rightarrow \MAG$ and the \emph{magma algebra functor} $\fK [-] \colon\MAG \rightarrow \ALG$. In particular, for a set~$X$, $M(X)$ is the magma of non-empty (non-associative) words in $X$, which are sometimes represented by means of \emph{binary trees} (which are rooted and proper, with leaves labelled by elements in~$X$). By a \defn{non-associative polynomial} $\varphi $ on a set~$X$ we mean an element of $\fK[M(X)]$. Such a polynomial is said to be a \defn{monomial} if it is a scalar multiple of an element in $M(X)$. For example, if $X=\lbrace x,y,z \rbrace$, then $xy+yx$, $x^2+yz$ and $x(yz)$ are polynomials in $X$ and only the last one is a monomial. The \defn{type} of a monomial $\varphi$ on a set $\lbrace x_1, \dots, x_n \rbrace$ is an $n$-tuple $(k_1, \dots , k_n)\in \N^n$ where $k_i$ is the number of times $x_i$ appears in $\varphi$. The \defn{degree} of $\varphi$ is $k_1 + \dots + k_n$. A polynomial is said to be \defn{homogeneous} if all its monomials are of the same type. A homogeneous polynomial is said to be \defn{multilinear} if it is of type $(1, \dots , 1)$.

\begin{definition}\label{Definition identity variety}
    A non-associative polynomial $\varphi=\varphi(x_1,\dots ,x_n)$ is called an \defn{identity} of an algebra $A$ if $\varphi(a_1, \dots , a_n)=0$ for all $a_1$, \dots, $a_n \in A$. We also say that~$A$ \emph{satisfies} the identity $\varphi$.

    Let $X$ be a set of variables and $I$ be a subset of $\fK[M(X)]$. The class of all $\fK$-algebras that satisfy all identities in $I$ is called the \defn{variety of $\fK$-algebras} determined by $I$. Given such an $I$, we say that a variety \defn{satisfies the identities in $I$} if all algebras in this variety satisfy the given identities. In particular, if the variety is determined by a set of homogeneous polynomials, then we say that the variety is \defn{homogeneous}, and if the variety is determined by a set of multilinear polynomials, then we say that it is \defn{operadic}.
\end{definition}

For the sake of simplicity, we may sometimes not write an identity as a polynomial but rather as an equality. For example, for commutativity, in order to express that the polynomial $xy-yx$ is an identity, we may write that $xy=yx$ (or $xy-yx=0$) is satisfied in the algebra under consideration. This abuse of terminology will make the reading easier.

\begin{remark}\label{Remark operadic}
    If the characteristic of the field $\fK$ is zero, then by a multilinearisation process, any variety of non-associative algebras over $\fK$ is operadic (see~\cite[Corollary~3.7]{Osborn}). For instance, the identity $xx=0$ for alternating algebras implies ${(y+z)(y+z)=0}$, so that $0=yy+yz+zy+zz=yz+zy$, which means that the algebras are anti-commutative; in characteristic zero, this multilinear identity determines the class of alternating algebras, which may thus be viewed as an operadic variety.

    The reason we call the varieties determined by multilinear polynomials ``operadic'' lies in Proposition 5.7.3 of the reference book~\cite{LodayValette-AlgebraicOperads} about \emph{algebraic operads}. In fact, any type of algebras (over a field of characteristic $0$) in the sense of operads whose identities are multilinear always determines an \emph{operad} whose algebras form the given variety.

    In some sense, ``operadic'' simply means ``expressible in a (symmetric) monoidal category'', as is the viewpoint in the categorical references~\cite{MacLane-CA,Kelly-Operads} on this subject. Following their terminology, our varieties determined by multilinear polynomials might instead be called \emph{PROP varieties of $\fK$-algebras}.
\end{remark}

\begin{examples}\label{Examples varieties}
    A $\fK$-algebra $A$ is said to be \defn{associative} if the polynomial $x(yz)-(xy)z$ is an identity of $A$, or equivalently if $a(bc)=(ab)c$ for all $a$, $b$, $c\in A$. Similarly, we say that $A$ is \defn{anti-associative} if it satisfies the identity $x(yz)+(xy)z$, \defn{commutative} if $xy-yx=0$, \defn{anti-commutative} if $xy+yx=0$, \defn{abelian} if $xy=0$ and \defn{trivial} if $x=0$.

    The variety of commutative associative $\fK$-algebras will be denoted by $\CA$. We write $\ANTI$ for the variety of anti-associative and anti-commutative $\fK$-algebras. The subvariety $\ABALG$ of $\ALG$ determined by the abelian algebras is isomorphic to the variety $\VECT$ of $\fK$-vector spaces. An algebra is trivial if and only if it is a singleton; the variety of such algebras is denoted $\TRIV$.

    We write $\LIE$ for the variety of \defn{Lie algebras} over $\fK$, which are anti-commutative and satisfy the \defn{Jacobi identity} $x(yz)+y(zx)+z(xy)$.

    All the above-mentioned examples are operadic varieties; let us give some non-operadic examples. The non-homogeneous variety $\BOOL_\fK$ of \emph{Boolean $\fK$-algebras} is the variety determined by associativity, commutativity and the identity $xx-x$. For a field $\fK$ of prime characteristic $p$ we denote by $\CAp$ the homogeneous variety of commutative associative $\fK$-algebras satisfying the additional identity $x^p=0$. Non-operadic varieties are the subject of Section~\ref{Section Non-Operadic Case}.
\end{examples}

\begin{remark}
    Since a variety of non-associative algebras is always a \emph{variety of $\Omega$-groups} in the sense of~\cite{Higgins}, it is a \emph{semi-abelian category}~\cite{Janelidze-Marki-Tholen}. This will be useful in Section~\ref{Higgins Commutator}.
\end{remark}

The commutative and anti-commutative cases are of critical importance, because of the following lemma:

\begin{lemma}\label{Lemma Degree Two}
    Let $\fK$ be a field. If $\cV$ is a variety of $\fK$-algebras satisfying a non-trivial homogeneous identity of degree $2$, then it must satisfy at least one of the identities $xy - yx$, $xy +yx$ or $xy$. In particular, if $\cV$ is non-abelian, then it is either a subvariety of the variety of commutative algebras, or a subvariety of the variety of anti-commutative algebras.
\end{lemma}
\begin{proof}
    A homogenous identity of degree $2$ is necessarily an element $\varphi(x,y)$ of the free $\fK$-algebra on two generators $x$ and $y$, so a polynomial of the form
    \[
        \lambda_1xx+\lambda_2xy+\lambda_3yx+\lambda_4yy.
    \]
    Homogeneity implies that either $\lambda_1xx=0$ or $\lambda_2xy+\lambda_3yx=0$. Note that the former implies the latter. If now either $\lambda_2=0$ or $\lambda_3=0$, then we see that $yx=0$ or $xy=0$ and $\cV$ is abelian. Otherwise, we write $\lambda=-\lambda_3/\lambda_2$ and $xy=-\lambda yx=\lambda^2xy$, so that $(1-\lambda^2)xy=0$. Hence either $\cV$ is abelian or $\lambda\in \{-1,1\}$.
\end{proof}


\section{Cosmash products}\label{Section Definition Cosmash}
In order to define \emph{cosmash associativity}, which is the central concept of this paper, in Section~\ref{Section Cosmash Associativity}, we need to explain what is a \emph{cosmash product}. In this section, we start with the \emph{binary cosmash product}, ask ourselves how we can extend it to a ternary one, and conclude that those two are specific cases (for $n\in\{2,3\}$) of an $n$-ary definition. In Section~\ref{Higgins Commutator} we shall see how, in the context of a so-called \emph{semi-abelian} category, this leads to a categorical approach to the \emph{Higgins commutator}. In this section, and throughout the rest of the paper unless otherwise stated, we work in a pointed category ($0=1$) with finite sums, finite products, and kernels.

\begin{notation}
    Let $\cC$ be a pointed category with finite sums and finite products and $X$, $Y$ two objects of $\cC$. Then there exists a canonical morphism
    \begin{align*}
        \Sigma_{X,Y}\coloneq\bigl(\begin{smallmatrix}
                                          1_X & 0 \\
                                          0 & 1_Y
                                      \end{smallmatrix}\bigr)\colon X + Y \rightarrow X \times Y.
    \end{align*}

\end{notation}

\begin{definition}\label{Definition Binary Cosmash}
    The \defn{binary cosmash product} $X\diamond Y$ of two objects $X$ and $Y$ in a pointed category with finite sums, finite products, and kernels is defined as the kernel
    \[
        \xymatrix{0 \ar[r] & X\diamond Y \ar@{{ |>}->}[r]^-{\iota_{X,Y}} & X + Y \ar[r]^-{\Sigma_{X,Y}} & X \times Y}
    \]
    of the morphism $\Sigma_{X,Y}$.
\end{definition}

This definition was first given (dually) by Carboni and Janelidze in~\cite{Smash} where they define the (binary) \emph{smash} product as the cokernel of $\Sigma_{X,Y}$. Independently, in~\cite{MM-NC}, Mantovani and Metere used the binary cosmash product as a so-called \emph{formal commutator} in a categorical approach to the Higgins commutator of~\cite{Higgins}. Later, in~\cite{HVdL}, Hartl and Van der Linden recovered it as a special case of a \emph{cross-effect}, using \emph{higher} cosmash products---i.e., $n$-ary cosmashes for $n\geq 2$, see Definition~\ref{nary cosmash}---to capture the higher commutators of~\cite{Higgins} as in Section~\ref{Higgins Commutator} below. See also~\cite{BeBou} for a complementary point of view.

\begin{examples}\label{Examples binary cosmash}
    \begin{enumerate}
        \item If the category $\cC$ is additive, then any cosmash product $X \diamond Y$ is trivial, because each comparison morphism $\Sigma_{X,Y}$ is an isomorphism. (We shall not worry here about potential non-existence of kernels in additive categories, since we are only ever taking kernels of monomorphisms.)
        \item In the category $\GP$ of groups, it is well known that the cosmash product of two objects $X$ and~$Y$ is the subgroup of $X+Y$ generated by words of the form $xyx^{-1}y^{-1}$ and $yxy^{-1}x^{-1}$, where $x\in X$ and $y\in Y$. We recall the argument: any element $w$ of in $X + Y$ can be written as
              \begin{align*}
                  w =x_1y_1 \cdots x_ny_n
              \end{align*}
              where $x_i\in X$ and $y_i\in Y$ for all $i\in\{1, \dots , n\}$. If now $w$ is $x$, $y$, $xy$, $yx$, $xyx'$ or $yxy'$, then its image under $\Sigma_{X,Y}$ is respectively $(x,1)$, $(1,y)$, $(x,y)$, $(x,y)$, $(xx',y)$ or $(x,yy')$. For such a $w$, belonging to the kernel $X\cosmash Y$ of $\Sigma_{X,Y}$ simply means that $w$ is trivial. If, however, $w=xyx'y'$ or $w=yxy'x'$, then $\Sigma_{X,Y}$ sends it to $(xx',yy')$, so that it belongs to the kernel of $\Sigma_{X,Y}$ precisely when $x'=x^{-1}$ and $y'=y^{-1}$. Now, a generic word $w=x_1y_1 \cdots x_ny_n$ of length $2n$ can be rewritten as
              \begin{align*}
                  w=(x_1y_1x_1^{-1}y^{-1}_1)y_1x_1x_2y_2 \cdots x_ny_n,
              \end{align*}
              where $x_1x_2$ is another element $x'$ of $X$. We can observe that the word $y_1x'y_2 \cdots x_ny_n$ is of length $2n-1$. Therefore, by an induction argument, we may write
              \begin{align*}
                  w=(x_1y_1x_1^{-1}y_1^{-1})(y_1x'y_1^{-1}x'^{-1})\cdots (x''y''x_ny_n)
              \end{align*}
              which is sent to $(x''x_n,y''y_n)$ by $\Sigma_{X,Y}$. Hence if $w\in X \diamond Y$, then $x''=x_n^{-1}$ and $y''=y_n^{-1}$, so that $w$ is a product of elements of the form $xyx^{-1}y^{-1}$ and $yxy^{-1}x^{-1}$ with $x\in X$ and $y\in Y$.
        \item In a variety $\cV$ over a field $\fK$, it is easy to see that for two algebras $X$ and~$Y$ in $\cV$, the coproduct $X+Y$ is the set of (equivalence classes with respect to the identities of $\cV$ and those of $X$ and $Y$ of) the (non-associative) polynomials with variables in $X$ and~$Y$. Hence the cosmash product $X \diamond Y$ is the subalgebra of $X+Y$ determined by such polynomials where each monomial contains at least one variable in $X$ and one in $Y$. Indeed, those and no others are sent by $\Sigma_{X,Y}$ to $(0_X,0_Y)$. For example, $xy$ is in $X\diamond Y$ but $x^2y + y^3$ is not.
        \item Section~\ref{Section Cosmash Is Tensor} explains why in the variety $\CA$, the cosmash product $X\diamond Y$ is the tensor product $X \otimes Y$. (This is not to be confused with the well-known fact that in the category of \emph{unitary} commutative associative algebras, the \emph{coproduct} is already the tensor product: see, for instance, \cite{MacLane:Homology}.)
    \end{enumerate}
\end{examples}

Next, we want to extend our definitions to $n$-ary cosmash products for any natural number $n\geq 2$. A first idea might be to define the ternary cosmash product of three objects $X$, $Y$, $Z$ in $\cC$ by iterating the binary cosmash product, which gives us $X\diamond(Y\diamond Z)$ or $(X\diamond Y)\diamond Z$. The problem with this approach is that those objects are not isomorphic in general, as shown for instance in Examples~\ref{Example Not Surjective} and Examples~\ref{Examples Injectivity}. Here, we may already notice that in the category $\ALG$, for some $x\in X$, $y\in Y$ and $z\in Z$, the element $x(yz)$ lies in $X\diamond (Y\diamond Z)$, yet it cannot be seen as an element of $(X\diamond Y) \diamond Z$, since $x$ is outside the brackets. Actually, the elements of $X\diamond (Y\diamond Z)$ are not exactly polynomials in the elements of $X$, $Y$ and~$Z$, but rather polynomials with variables in $X$ and in $Y \diamond Z$.

In general, there is not even a canonical \emph{morphism} from $X\diamond(Y\diamond Z)$ to $(X\diamond Y)\diamond Z$. The idea is now to consider an \emph{unbiased} ternary cosmash product $X\diamond Y\diamond Z$ which may act as a ``common receptacle'' for both, as in the diagram
\[
    \xymatrix{& X\diamond Y\diamond Z\\
    X\diamond(Y\diamond Z) \ar[ru]^-{\Phi_{X,Y,Z}}  && (X\diamond Y)\diamond Z. \ar[lu]_-{\Psi_{X,Y,Z}}}
\]
Asking that both $\Phi_{X,Y,Z}$ and $\Psi_{X,Y,Z}$ are isomorphisms then provides us with a way to express that $X\diamond(Y\diamond Z)$ and $(X\diamond Y)\diamond Z$ (and $X\diamond Y\diamond Z$) are isomorphic.

A first idea towards a definition of a ternary cosmash product $X\diamond Y\diamond Z$ might be to take the kernel of the morphism
\[
    \Bigl(\begin{smallmatrix}
            1_X & 0 & 0 \\
            0 & 1_Y & 0 \\
            0 & 0 & 1_Z
        \end{smallmatrix}\Bigr)\colon X+Y+Z \to X \times Y \times Z.
\]
However, in the category of groups, for example, the word $xyx^{-1}y^{-1}$ will be an element of this kernel. This does not agree with what a commutator word in $X$, $Y$ and $Z$ is supposed to be, since it contains no elements of $Z$. Instead, we follow the approach of~\cite{Smash, HVdL,BeBou}:

\begin{definition}
    Let $X$, $Y$ and $Z$ be three objects of $\cC$ and consider the canonical morphism
    \begin{align*}
        \Sigma_{X,Y,Z} \coloneq
        \Bigl(\begin{smallmatrix}
                      \iota_X & \iota_Y & 0 \\
                      \iota_X & 0 & \iota_Z \\
                      0 & \iota_Y & \iota_Z
                  \end{smallmatrix}\Bigr)
        \colon X + Y + Z \rightarrow (X+Y) \times (X+Z) \times (Y+Z).
    \end{align*}
    The \defn{ternary cosmash product} $X\diamond Y\diamond Z$ of $X$, $Y$ and $Z$ is defined as the kernel of $\Sigma_{X,Y,Z}$. The canonical inclusion is denoted $\iota_{X,Y,Z}\colon X\diamond Y\diamond Z\to X+Y+Z$.
\end{definition}

With this definition, in the case of groups we cannot have elements such as $xyx^{-1}y^{-1}$ in the ternary cosmash product. However, we have elements such as $xyx^{-1}y^{-1}zyxy^{-1}x^{-1}z^{-1}$ in $X\diamond Y\diamond Z$ for some $x\in X$, $y\in Y$ and $z\in Z$. In the category $\ALG$, elements of $X\diamond Y\diamond Z$ will be polynomials where each monomial contains variables in $X$, $Y$ and $Z$. The same holds for any subvariety of~$\ALG$. In any additive category, the ternary cosmash product is the zero object, since there $\Sigma_{X,Y,Z}$ is a monomorphism.

\subsection{Another point of view}
The following well-known lemma will be useful for us; its proof follows immediately from the universal properties involved. It allows us to obtain the alternative construction of the ternary cosmash product of Lemma~\ref{Lemma ternary cosmash via cross effect}.

\begin{lemma}\label{Lemma Kernel and Intersection}
    For some $n\geq 2$, consider objects $Y_1$, \dots, $Y_n$ and a morphism $(x_1,\dots, x_n)\colon X\to Y_1\times \cdots \times Y_n$. Then the kernel of $(x_1,\dots, x_n)$ is the intersection $\bigwedge_{i=1}^{n}\ker(x_i)$ of the kernels of all the $x_i$.\noproof
\end{lemma}

\begin{lemma}\label{Lemma ternary cosmash via cross effect}
    The ternary cosmash product can be obtained out of the binary cosmash as a so-called \emph{cross-effect}~\cite{Baues-Pirashvili,Hartl-Vespa, HVdL}. In fact, $X\cosmash Y\cosmash Z$ is a kernel of
    \[
        \bigl(\begin{smallmatrix}
                \langle 1_X, 0\rangle\cosmash Z \\
                \langle 0, 1_Y\rangle\cosmash Z
            \end{smallmatrix}\bigr) \colon (X+Y)\cosmash Z \to (X\cosmash Z) \times (Y\cosmash Z).
    \]
\end{lemma}
\begin{proof}
    The upper row in Figure~\ref{Figure Ternary via Cross Effect}
    \begin{figure}
        {$\xymatrix@C=4em{W \pullback \ar@{{ |>}->}[d] \ar@{{ |>}->}[r] & (X+Y)\cosmash Z \ar@{{ |>}->}[d]_-{\iota_{X+Y,Z}} \ar[r]^-{\bigl(\begin{smallmatrix}
                        \langle 1_X, 0\rangle\cosmash Z \\
                        \langle 0, 1_Y\rangle\cosmash Z
                    \end{smallmatrix}\bigr)} & (X\cosmash Z) \times (Y\cosmash Z) \ar@{{ |>}->}[d]\\
                \cdot \ar[d] \ar@{{ |>}->}[r] & X+Y+Z \ar[r]^-{\bigl(\begin{smallmatrix}
                        \iota_X & 0 & \iota_Z \\
                        0 & \iota_Y & \iota_Z
                    \end{smallmatrix}\bigr)} \ar[d]_-{\bigl(\begin{smallmatrix}
                        \iota_X & \iota_Y & 0 \\
                        0 & 0 & 1_Z
                    \end{smallmatrix}\bigr)} & (X+Z)\times (Y+Z) \ar[d] \\
                \cdot \ar@{{ |>}->}[r] & (X+Y)\times Z \ar[r] & (X\times Z)\times (Y\times Z)}$}
        \caption{Ternary cosmash $W\cong X\cosmash Y\cosmash Z$ as cross-effect}\label{Figure Ternary via Cross Effect}
    \end{figure}
    represents $W$ as a kernel of the morphism $\bigl(\begin{smallmatrix}
                \langle 1_X, 0\rangle\cosmash Z \\
                \langle 0, 1_Y\rangle\cosmash Z
            \end{smallmatrix}\bigr)$ : indeed, by Lemma~\ref{Lemma Kernel and Intersection}, the object $W$ is the intersection of the kernels of $\langle \iota_X, \iota_Y, 0\rangle$, $\langle \iota_X, 0, \iota_Z\rangle$, $\langle 0, \iota_Y, \iota_Z\rangle$ and $\langle 0,0,1_Z\rangle$, which is the same as the kernel of $\Sigma_{X,Y,Z}$ because the morphism $\langle 0,0,1_Z\rangle$ factors through $\langle \iota_X, 0, \iota_Z\rangle$.
\end{proof}

See also Definition~\ref{Definition of cross-effect} and Remark~\ref{Comparison map via cross effect} below. This viewpoint will be helpful in Section~\ref{Section Cosmash Is Tensor} when we calculate the cosmash product in the variety $\CA$, and also in Section~\ref{Section Injectivity} where it will allow us to obtain Lemma~\ref{Lemma Comparison Injective Implication}.

\subsection{Higher cosmash products}
The binary and ternary cosmash products are instances of a general definition:
\begin{definition}\cite{Smash, HVdL}\label{nary cosmash}
    For $n\geq 2$, let $X_1$, \dots, $X_n$ be objects, and consider the canonical morphism
    \begin{align*}
        \Sigma_{X_1,\dots,X_n}\colon X_1+ \dots + X_n \rightarrow \displaystyle \prod_{k=1}^n\coprod_{\substack{j=1 \\j\neq k}}^n X_j
    \end{align*}
    determined by
    \begin{align*}
        \pi_{\coprod_{\substack{j\neq k}}X_j}\circ \Sigma_{X_1,\dots,X_n} \circ \iota_{X_l} = \begin{cases}
                                                                                                  \iota_{X_l} & \text{if $l \neq k$} \\
                                                                                                  0           & \text{if $l=k$.}
                                                                                              \end{cases}
    \end{align*}
    The \defn{$n$-ary cosmash product} $X_1\diamond \dots \diamond X_n$ of $X_1$, \dots, $X_n$ is the kernel of $\Sigma_{X_1,\dots,X_n}$. We write $\iota_{X_1, \dots, X_n}\colon X_1\diamond \dots \diamond X_n\to X_1+ \dots + X_n$ for the canonical inclusion.
\end{definition}

\begin{remark}\label{symmetry of cosmash}
    Using that products and coproducts are symmetric, it is easy to see that also Definition~\ref{nary cosmash} is symmetric in the variables $X_i$ and thus, for any permutation $\sigma \in \mathfrak{S}_n$, we have that $X_1\diamond \dots \diamond X_n\cong X_{\sigma(1)}\diamond \dots \diamond X_{\sigma(n)} $. In some sense, it is precisely this fundamental feature of the cosmash product which will limit the scope of our main result to the context of (anti-)commutative algebras.
\end{remark}

Also higher-order cosmash products admit a cross-effect interpretation (as in Lemma~\ref{Lemma ternary cosmash via cross effect}). In order to be sufficiently precise for the use we make of this in Section~\ref{Section Injectivity}, we recall the following definition from~\cite{HVdL}.

\begin{definition}\label{Definition of cross-effect}
    Let $F\colon {\cC} \to {\cD}$ be a functor from a pointed category with finite sums ${\cC}$ to a pointed category with finite products and kernels ${\cD}$. The \defn{second cross-effect of~$F$} is the functor $\cre_2(F)\colon \cC\times \cC \to \cD$ defined as follows: it sends a pair of objects $(X,Y)$ in $\cC\times\cC$ to the object $\cre_2(F)(X,Y)$ which is the kernel of
    \[
        \bigl(\begin{smallmatrix}
                F(\langle 1_X,0\rangle)\\
                F(\langle0,1_Y\rangle)
            \end{smallmatrix}\bigr)
        \colon F(X +Y) \to F(X)\times F(Y)
    \]
    ---which extends to morphisms in the obvious manner.
\end{definition}

\begin{examples}\label{Examples Cross-Effects}
    We let $\cC$ be a pointed category with finite coproducts, products and kernels.
    \begin{enumerate}
        \item It is immediately clear that when $F=1_{\cC}$ we regain Definition~\ref{Definition Binary Cosmash}, so that $X \diamond Y=\cre_2(1_{\cC})(X,Y)$.
        \item Lemma~\ref{Lemma ternary cosmash via cross effect} tells us that $X\cosmash Y\cosmash Z\cong \cre_2((-)\cosmash Z)(X,Y)$; in fact, the functor $\cC\times\cC\to \cC$ which sends a couple of objects $(X,Y)$ to $X\cosmash Y\cosmash Z$ is the second cross-effect $\cre_2((-)\cosmash Z)$ of the functor $(-)\cosmash Z\colon \cC\to \cC$. By symmetry, also $X\cosmash Y\cosmash Z\cong \cre_2(X\cosmash (-))(Y,Z)$.
        \item Lemma~2.20 in the first arXiv version of~\cite{HVdL} implies that for any $X$, $Y$, $Z$ and $W$,
              \[
                  \cre_2(X\cosmash Y\cosmash (-))(Z,W)\cong X\cosmash Y\cosmash Z\cosmash W\cong \cre_2((-)\cosmash Z\cosmash W)(X,Y).
              \]
              The proof is a straightforward variation on Lemma~\ref{Lemma ternary cosmash via cross effect}; see~\cite{RVV3} for further details.
    \end{enumerate}
\end{examples}

\begin{remark}\label{Remark Functor cr}
    Note that the cross-effect operation $\cre_2$ may be seen as a functor $\FUN(\cC,\cD)\to \FUN(\cC^2,\cD)$ from $\FUN(\cC,\cD)$, the category of functors from $\cC$ to $\cD$, to $\FUN(\cC^2,\cD)$, the category of functors from $\cC^2$ to $\cD$. It follows immediately from its construction in Definition~\ref{Definition of cross-effect} that $\cre_2$ preserves natural monomorphisms.
\end{remark}

\section{The Higgins commutator}\label{Higgins Commutator}

In this section we work in a \defn{semi-abelian} category~\cite{Janelidze-Marki-Tholen}: a pointed, Barr-exact, Bourn-protomodular category with binary coproducts. From these axioms it follows that any semi-abelian category is finitely (co)complete. Non-expert readers need not to know too much about semi-abelian categories to understand the basic use we make of them here, and can just think of the context as being weaker than abelian, yet including categories such as the category $\GP$ of groups, the category $\RNG$ of rings without unit, or any variety of non-associative $\fK$-algebras (as defined in Section~\ref{Section Varieties}).

If the objects considered in Examples~\ref{Examples binary cosmash} are subobjects $K$ and $L$ of a common object $X$ in the category~$\cC$, then the cosmash product $K \diamond L$ consists of ``formal commutator words'' which lie in the coproduct $K+L$, but are not concrete elements of $X$. By taking a suitable image factorisation, this can be used to define the commutator of $K$ and $L$ as a subobject of $X$.

\begin{definition}\cite{MM-NC, HVdL}
    For an object $X$, consider subobjects $K$, $L\leq X$ respectively represented by monomorphisms $k\colon K\to X$ and $l\colon L\to X$. We define the \defn{Higgins commutator} $[K,L]$ of $K$ and $L$ as the regular image of the composite $\langle k,l\rangle \circ \iota_{K,L}$ as in the diagram
    \[
        \xymatrix{K\diamond L \ar[r]^-{\iota_{K,L}} \ar@{-{ >>}}[d] & K +L \ar[d]^{\langle k,l\rangle} \\
        {[K,L]} \ar@{{ >}->}[r] & X.}
    \]
\end{definition}

We observe that, in the case of groups, this gives the usual commutator of group theory. In the case of non-associative algebras, one can be tempted to expect that elements such as $kl-lk$ lie in $[K,L]$---which they do---but not elements such as~$kl$. However, the ``universal'' abelian subcategory of any variety of non-associative $\fK$-algebras is the one defined by $kl=0$, the variety $\ABALG$ which, as already explained in Section~\ref{Section Varieties}, is isomorphic to the category of $\fK$-vector spaces. So here, rather than characterising commutativity, it characterises abelianness.

As before, one can define the \defn{$n$-ary Higgins commutator} $[K_1, \dots, K_n]$ of a finite collection $K_i\leq X$ of subobjects of an object $X$ represented by monomorphisms $(k_i\colon K_i \rightarrow X)_{1\leq i\leq n}$ for $n\geq 2$ as the regular image
\[
    \xymatrix{K_1 \diamond \dots \diamond K_n \ar[r]^-{\iota_{K_1, \dots, K_n}} \ar@{-{ >>}}[d] & K_1 + \dots + K_n \ar[d]^-{\langle k_1, \dots, k_n\rangle}\\
    {[K_1, \dots , K_n]} \ar@{{ >}->}[r]_-{[k_1,\dots,k_n]} & X}
\]
of the composite $\langle k_1, \dots, k_n\rangle \circ \iota_{K_1, \dots, K_n}$: this is the approach taken in~\cite{HVdL} towards a categorical version of the definition in~\cite{Higgins}.

\section{Cosmash associativity}\label{Section Cosmash Associativity}

Thanks to Section~\ref{Section Definition Cosmash}, we are now able to explain what cosmash associativity is, and see why this condition does not hold for some ``classical'' algebraic objects. Section~\ref{Section Cosmash Is Tensor} deals with a specific variety of algebras which does satisfy the cosmash associativity condition.

We again let $\cC$ be a pointed category with finite sums, finite products, and kernels. Let us observe that for any $X$, $Y$, $Z\in \cC$, there exists a canonical comparison map from $X\diamond (Y\diamond Z)$ (or from $(X\diamond Y)\diamond Z$) to the ternary cosmash product $X \diamond Y\diamond Z$. This can be seen through the diagram
\[
    \xymatrix{0 \ar[r] & X\diamond (Y\diamond Z) \ar@{{ |>}->}[r]^-{\iota_{X,Y\diamond Z}} \ar@{.>}[d]_-{\Phi_{X,Y,Z}} & X + (Y\diamond Z) \ar[d]^-{1_X + \iota_{Y,Z}} \ar[r]^-{\Sigma_{X, Y\diamond Z}} & X \times (Y\diamond Z) \ar[d]^-{(\iota_X,\iota_X)\times \iota_{Y,Z}} \\
    0 \ar[r] & X \diamond Y \diamond Z \ar@{{ |>}->}[r]_-{\iota_{X,Y,Z}} & X + (Y + Z) \ar[r]_-{\Sigma_{X,Y,Z}} & (X+Y) \times (X+ Z) \times (Y+Z)}
\]
whose rows are exact sequences. Since the right hand square of the diagram commutes, by the universal property of kernels there exists a unique morphism $\Phi_{X,Y,Z}$ from $X\diamond (Y\diamond Z)$ to $X\diamond Y\diamond Z$ making the left hand square commute. Likewise, a canonical morphism $\Psi_{X,Y,Z}\colon (X\diamond Y)\diamond Z\to X\diamond Y\diamond Z$ may be constructed. Remark~\ref{Comparison map via cross effect} below provides alternate constructions for these comparison maps. As we shall explain, they need not be isomorphisms in general---whence the following definition.

\begin{definition}\cite{Smash}\label{Def Assoc Cosmash}
    A pointed category with finite sums, finite products, and kernels $\cC$ is said to be \defn{cosmash associative}, or \defn{satisfies cosmash associativity}, whenever for every three objects $X$, $Y$, $Z$ in~$\cC$ the two canonical comparison maps
    \[
        \xymatrix{& X\diamond Y\diamond Z\\
        X\diamond(Y\diamond Z) \ar[ru]^-{\Phi_{X,Y,Z}}_-\cong \ar@{.>}[rr]_-\cong && (X\diamond Y)\diamond Z \ar[lu]_-{\Psi_{X,Y,Z}}^-\cong}
    \]
    are invertible, which yields the dotted comparison isomorphism.
\end{definition}

This definition is in some sense redundant. Indeed, suppose that for any $X$, $Y$, $Z\in \cC$, the comparison map $\Phi_{X,Y,Z}\colon X\diamond(Y\diamond Z)\rightarrow X\diamond Y\diamond Z$ is an isomorphism. Since by Remark~\ref{symmetry of cosmash} we have that $X\diamond (Y\diamond Z)\cong (Y\diamond Z)\diamond X $ and $X\diamond Y \diamond Z \cong Y \diamond Z \diamond X $, also the comparison map $\Psi_{X,Y,Z}\colon(Y\diamond Z)\diamond X\rightarrow Y\diamond Z\diamond X$ is an isomorphism. This shows that in order to prove the cosmash associativity of a category $\cC$, it suffices to check the property for any one of the two canonical comparison maps. Note that we are already using this redundancy, in some sense, since we never mention the third canonical comparison map $Y\cosmash (X\cosmash Z)\to X\diamond Y\diamond Z$.

\begin{examples}\label{examples and counterexamples of cosmash asso}
    \begin{enumerate}
        \item Any additive category is cosmash associative, since all cosmash products are trivial. In particular, the varieties $\ABALG$ and $\TRIV$ are so.
        \item The variety $\CA$ of associative and commutative $\fK$-algebras is cosmash associative, since the cosmash product corresponds to the tensor product. See Section~\ref{Section Cosmash Is Tensor} for further details.
        \item In~\cite{Smash}, examples are given of (non-algebraic) categories which satisfy the dual condition: the so-called \emph{smash product} is associative.
        \item In the category $\GP$ of groups, cosmash products are not associative in general, essentially because commutators are not associative. Indeed, if they were, it would be possible to deduce for any triple of subgroups $K$, $L$, $M\leq X$ that $[[K,L],M]=[K,L,M]$ as subobjects of $X$. This does, however, contradict the fact that $[K,[L,M]]$ need not be contained in $[[K,L],M]$. We may, for instance, take $K=L=\Z$ viewed as included on the left in $X=M=\Z+\Z$. Here the latter commutator $[[\Z,\Z],\Z+\Z]$ vanishes since $\Z$ is abelian, while the former $[\Z,[\Z,\Z+\Z]]$ is highly non-trivial.
        \item A similar reasoning works to prove that cosmash products of Lie algebras need not be associative. A more direct argument will be given below (Example~\ref{Example Lie fails surjectivity}), where the result follows from an analysis of the comparison morphism occurring in Definition~\ref{Def Assoc Cosmash}.
    \end{enumerate}
\end{examples}

\begin{remark}\label{Remark Removing Brackets}
    Proposition 2.21 in~\cite{HVdL} tells us that \emph{removing brackets in a commutator enlarges the object}, which can be written as
    \begin{align*}
        [[K_1,\dots,K_i],K_{i+1}, \dots , K_n] \leq [K_1,\dots,K_n]
    \end{align*}
    for each $2\leq i\leq n-1$, where the $K_j$ are subobjects of an object $X$. This does not hold for cosmash products. In particular, the comparison maps in Definition~\ref{Def Assoc Cosmash} need not be injective in general---see Section~\ref{Section Injectivity}. In fact, they need not be surjections either, as explained in Section~\ref{Section Surjectivity}. Both conditions make sense on their own, and will be explored in some detail in what follows.
\end{remark}

\begin{remark}
    One might be tempted to define cosmash associativity as the condition that $X\diamond(Y\diamond Z)$ and $(X\diamond Y) \diamond Z$ are isomorphic objects. This idea does not seem as fruitful as the approach taken here, essentially because there exists no canonical map between these two objects. Furthermore, our methodology is compatible with the underlying so-called \emph{lax-monoidal} structure of the cosmash product, which we plan to further investigate in subsequent work~\cite{RVV3}.
\end{remark}

\begin{remark}\label{Comparison map via cross effect}
    Lemma~\ref{Lemma ternary cosmash via cross effect} gives us an alternative view on the construction of the comparison maps: for instance, $\Psi_{X,Y,Z}$ is the unique dotted factorisation in the diagram
    \[
        \xymatrix@C=4em{& (X\cosmash Y)\cosmash Z \ar@{.>}[d]_-{\Psi_{X,Y,Z}} \ar[rd]^-{\iota_{X,Y}\cosmash Z} \\
        0 \ar[r] & X\cosmash Y\cosmash Z \ar@{{ |>}->}[r] & (X+Y)\cosmash Z \ar[r]^-{\bigl(\begin{smallmatrix}
                \langle 1_X, 0\rangle\cosmash Z \\
                \langle 0, 1_Y\rangle\cosmash Z
            \end{smallmatrix}\bigr)} & (X\cosmash Z) \times (Y\cosmash Z).}
    \]
    Cosmash associativity may thus be viewed as a kind of left exactness property of the functor $(-)\cosmash Z\colon \cC\to \cC$, since it means that the morphism ${\iota_{X,Y}\cosmash Z}$ is the intersection of the kernels of $\langle 1_X, 0\rangle\cosmash Z$ and $\langle 0, 1_Y\rangle\cosmash Z$.
\end{remark}

\subsection{A linear independence result}\label{Subsection Linear Independence}
We now take the time to analyse a rather technical consequence of cosmash associativity, a key aspect of the structure of the cosmash product of three free algebras which turns out to be important in what follows. Let $\cV$ be a variety of $\fK$-algebras which satisfies no non-trivial identities of degree smaller than or equal to $2$. We assume that $A$, $B$ and $C$ are free $\cV$-algebras on a single generator written $a$, $b$ and $c$, respectively.

We start by studying the structure of the cosmash product $A\cosmash B$. We notice that, since $\cV$ has no identities of degree lower than $3$, as a vector space, the algebra $A\cosmash B$ may be written as a direct sum $\fK[\{ab,ba\}]\oplus V$ where $\fK[\{ab,ba\}]$ denotes the free vector space on the monomials $ab$ and $ba$, and all monomials in $V$ have strictly larger degree. Since the multiplication on $A\cosmash B$ has an image which is contained in $V$, the identities of $\cV$ leave the space $\fK[\{ab,ba\}]$ untouched.

Likewise, the free algebra $C$ may be written as a direct sum $\fK[\{c\}]\oplus W$ where $\fK[\{c\}]$ denotes the free vector space on the monomial $c$, and all monomials in $W$ have strictly larger degree.

Item (3) of Examples~\ref{Examples binary cosmash} essentially tells us that the coproduct $(A\cosmash B) + C$ has, for its underlying vector space, a sum of tensor products
\[
    (A\cosmash B) \oplus C \oplus \bigl((A\cosmash B) \otimes C\bigr) \oplus \bigl(C\otimes (A\cosmash B)\bigr) \oplus \bigl(C\otimes (A\cosmash B) \otimes C\bigr) \oplus \cdots
\]
quotiented by the identities of $\cV$. Now
\begin{multline*}
    \bigl((A\cosmash B) \otimes C\bigr) \oplus \bigl(C\otimes (A\cosmash B)\bigr)\\
    \cong
    \bigl((\fK[\{ab,ba\}]\oplus V) \otimes (\fK[\{c\}]\oplus W)\bigr) \oplus \bigl((\fK[\{c\}]\oplus W)\otimes (\fK[\{ab,ba\}]\oplus V)\bigr),
\end{multline*}
where we may observe that those elements of its subset
\begin{align*}
    D \coloneq & \;\bigl(\fK[\{ab,ba\}]\otimes \fK[\{c\}]\bigr)\oplus \bigl(\fK[\{c\}]\otimes\fK[\{ab,ba\}]\bigr) \\
    \cong      & \;\fK[\{(ab)c,(ba)c,c(ab),c(ba)\}]
\end{align*}
which lie in the image of the algebra multiplication arise from multiplying an element of $\fK[\{ab,ba\}]$ with an element of $\fK[\{c\}]$. Hence, the fact that $\cV$ has no equations of degree $2$ ensures that these terms are left untouched by the above-mentioned quotient. We see that the underlying vector space of $(A\cosmash B) + C$ is $D\oplus X$ where $X$ is a quotient of the direct sum of
\[
    (A\cosmash B) \oplus C \oplus \bigl(C\otimes (A\cosmash B) \otimes C\bigr) \oplus \cdots
\]
with
\begin{multline*}
    (\fK[\{ab,ba\}] \otimes W) \oplus (V\otimes \fK[\{c\}])\oplus (V\otimes W)\\ \oplus (\fK[\{c\}]\otimes V)\oplus (W\otimes \fK[\{ab,ba\}])\oplus (W\otimes V).
\end{multline*}
In particular, $D$ is a subspace of $(A\cosmash B) + C$, which proves the independence in $(A\cosmash B) + C$ of the set of polynomials $\{(ab)c,(ba)c,c(ab),c(ba)\}$. Since $(A\cosmash B) \cosmash C$ is a subobject of $(A\cosmash B) + C$, this gives us:

\begin{lemma}\label{Lemma Independence}
    Let $\cV$ be a variety of $\fK$-algebras which satisfies no non-trivial identities of degree smaller than or equal to $2$. Let $A$, $B$ and $C$ be free algebras on a single generator written $a$, $b$ and $c$, respectively. Then all the ``degree $3$ monomials'' in $(A\cosmash B) \cosmash C$ are linearly independent: if
    \[
        \lambda_1(ab)c+\lambda_2(ba)c+\lambda_3c(ab)+\lambda_4c(ba)=0
    \]
    in $(A\diamond B)\diamond C$, then $\lambda_i=0$ for all $1\leq i\leq 4$.\noproof
\end{lemma}

\begin{remark}\label{Remark Independence Variations}
    The above procedure is easily modified to yield strengthenings of, and variations on, Lemma~\ref{Lemma Independence}. For instance, the set of monomials
    \[
        \{(ab)\omega(c),(ba)\omega(c),\omega(c)(ab),\omega(c)(ba)\}
    \]
    where $\omega(c)$ is any non-zero element of $\fK[\{c\}]$ is still linearly independent in the algebra ${(A\cosmash B)\cosmash C}$. As a consequence, if $E=\fK[\{e_1,\dots,e_n\}]$ is free on the set of generators $\{e_1,\dots,e_n\}$, then the surjective algebra morphism $f\colon E\to C\colon e_i\mapsto c$ induces the algebra morphism
    \[
        (1_A\cosmash 1_B) \cosmash f\colon (A\cosmash B) \cosmash E\to (A\cosmash B) \cosmash C.
    \]
    If now
    \[
        \lambda_1(ab)\phi(e_1,\dots,e_n)+\lambda_2(ba)\phi(e_1,\dots,e_n)+\lambda_3\phi(e_1,\dots,e_n)(ab)+\lambda_4\phi(e_1,\dots,e_n)(ba)
    \]
    is zero in $(A\cosmash B) \cosmash E$, then
    \[
        \lambda_1(ab)\phi(c,\dots,c)+\lambda_2(ba)\phi(c,\dots,c)+\lambda_3\phi(c,\dots,c)(ab)+\lambda_4\phi(c,\dots,c)(ba)=0
    \]
    in $(A\cosmash B) \cosmash C$, so that $\lambda_i=0$ for all $1\leq i\leq 4$. In particular, the monomials $(ab)\phi(e_1,\dots,e_n)$, $(ba)\phi(e_1,\dots,e_n)$, $\phi(e_1,\dots,e_n)(ab)$ and $\phi(e_1,\dots,e_n)(ba)$ are linearly independent in $(A\cosmash B) \cosmash E$, for any non-zero element $\phi(e_1,\dots,e_n)$ of $E$.

    This may be further generalised by considering arbitrary non-zero elements of $\fK[\{ab,ba\}]$ instead of just $ab$ or $ba$.
\end{remark}

\section{Cosmash and tensor}\label{Section Cosmash Is Tensor}

The aim of this short section is to recall from~\cite{Smash} that in the operadic variety $\CA$, the cosmash product is the tensor product over $\fK$. It follows that here, cosmash products are associative. We prove the result in full detail, so that it may serve as a basis for similar results in Section~\ref{Section Non-Operadic Case}.

In the variety $\CA$, the product of two algebras $A$ and $B$ is just the cartesian product $A\oplus B$ of the underlying vector spaces, equipped with the pointwise multiplication. It is also well known that the coproduct of $A$ and $B$ has
\[
    A \oplus (A \otimes B) \oplus B
\]
for its underlying vector space, equipped with the multiplication determined by
\begin{multline*}
    (a,a'\otimes b',b)(c,c'\otimes d',d)\\
    = (ac, ac'\otimes d'+a\otimes d+a'c\otimes b'+a'c'\otimes b'd'+a'\otimes b'd+c\otimes b+c'\otimes bd',bd)\,.
\end{multline*}
Indeed, given any two morphisms $f\colon A\to C$ and $g\colon B\to C$, the map
\[
    A \oplus (A \otimes B) \oplus B\to C\colon \Bigl(a,\sum_i a_i\otimes b_i,b\Bigr)\mapsto f(a)+\sum_i f(a_i)g(b_i)+g(b)
\]
is the unique algebra morphism which restricts to $f$ and $g$ via the inclusion of $A$ on the left and of $B$ on the right. Using the commutativity of the algebras in $\CA$, it is easily seen that this agrees with the description in item (3) of Examples~\ref{Examples binary cosmash}. In particular, the map sending $a\in A$ to $(a,0,0)$, $b\in B$ to $(0,0,b)$ and $ab$ to $(0, a\otimes b,0)$ determines an algebra isomorphism from the coproduct as described in Examples~\ref{Examples binary cosmash} to $A \oplus (A \otimes B) \oplus B$ with the given multiplication.

\begin{proposition}\label{Proposition Cosmash is Tensor}
    In the variety $\CA$, $A \diamond B \cong A \otimes B$ via the isomorphism
    \[
        A \otimes B \to A \diamond B\colon a\otimes b \mapsto ab\,.
    \]
\end{proposition}
\begin{proof}
    It suffices to notice that the comparison morphism
    \[
        \langle (1_A,0),(0,1_B)\rangle\colon A \oplus (A \otimes B) \oplus B\to A\oplus B
    \]
    sends an element $(a,a'\otimes b',b)$ to $(a,0)+ (a',0)(0,b')+(0,b)=(a,b)$, so that its kernel is the tensor product $A \otimes B$. The form of the isomorphism in terms of description in item (3) of Examples~\ref{Examples binary cosmash} follows right away from the description of the coproduct recalled above.
\end{proof}

\begin{proposition}\label{Proposition Ternary Cosmash is Tensor}
    In the variety $\CA$ we have $A \diamond B\diamond C \cong (A \otimes B)\otimes C$ for all algebras $A$, $B$, $C$.
\end{proposition}
\begin{proof}
    Via Proposition~\ref{Proposition Cosmash is Tensor}, the exact sequence in Lemma~\ref{Lemma ternary cosmash via cross effect} becomes
    \[
        \xymatrix{0 \ar[r] & A\cosmash B\cosmash C \ar@{{ |>}->}[r] & (A+B)\otimes C \ar[r] & (A\otimes C) \times (B\otimes C)}
    \]
    which says that the vector space underlying $A\cosmash B\cosmash C$ is the kernel of the linear map
    \[
        (A \oplus (A \otimes B) \oplus B)\otimes C\cong (A \otimes C)\oplus ((A \otimes B)\otimes C) \oplus (B\otimes C) \to (A\otimes C) \oplus (B\otimes C)\,,
    \]
    which is necessarily isomorphic to $(A \otimes B)\otimes C$. It is clear that the algebra structures agree with this.
\end{proof}

\begin{corollary}\label{Corollary Associative Cosmash}
    The variety $\CA$ has an associative cosmash product, which is part of a symmetric monoidal structure whose unit is the field $\fK$.\noproof
\end{corollary}

\begin{remark}
    This works as well when $\fK$ is a commutative ring with unit; we regain, for instance, the example of commutative rings (= commutative associative $\Z$-algebras) considered in~\cite{Smash}.
\end{remark}

\section{Surjectivity of the comparison maps}\label{Section Surjectivity}

Let $\cV$ be a homogeneous variety of algebras over a field $\fK$. Let $A$, $B$ and $C$ be free algebras in $\cV$, respectively generated by elements $a$, $b$ and $c$. Suppose the map $\Phi_{A,B,C} \colon A \diamond (B \diamond C) \to A \diamond B \diamond C$ is surjective.
Then the polynomial $(ab)c$ lies in the image of $\Phi_{A,B,C}$. So some polynomial
\[
    \LLL_1 a(bc) + \LLL_2 a(cb) + \LLL_3 (bc)a + \LLL_4 (cb)a + t(a,b,c)
\]
exists in $A \diamond (B \diamond C)$ which is mapped to $(ab)c$ by $\Phi_{A,B,C}$. Here $\LLL_1$, $\ldots$, $\LLL_4 \in \fK$ and $t(a,b,c)$ is a polynomial in $a$, $b$ and $c$ which does not contain any multilinear monomial of degree $3$. This implies that in $A+B+C$, we have
\[
    (ab)c = \LLL_1 a(bc) + \LLL_2 a(cb) + \LLL_3 (bc)a + \LLL_4 (cb)a + t(a,b,c)\,.
\]
Since $A+B+C$ is free, this expression is a linear combination of identities of $\cV$. By the homogeneity of those identities, necessarily also
\[
    (ab)c = \LLL_1 a(bc) + \LLL_2 a(cb) + \LLL_3 (bc)a + \LLL_4 (cb)a
\]
is an identity of $\cV$. We can think of this as an identity which ``pulls the left factor $a$ out of the parentheses'' in the expression $(ab)c$. Doing the same for the polynomial $(ba)c$ gives us the identity
\[
    (ba)c = \LLR_1 a(bc) + \LLR_2 a(cb) + \LLR_3 (bc)a + \LLR_4 (cb)a \,,
\]
which pulls the right factor out of the parentheses. The polynomials $c(ab)$ and $c(ba)$ net us the identities
\[
    c(ab) = \LRL_1 a(bc) + \LRL_2 a(cb) + \LRL_3 (bc)a + \LRL_4 (cb)a
\]
and
\[
    c(ba) = \LRR_1 a(bc) + \LRR_2 a(cb) + \LRR_3 (bc)a + \LRR_4 (cb)a \,,
\]
which allows to pull the factors out when the parentheses are located on the right.

\begin{proposition}\label{Proposition Lambda Rules}
    A homogeneous variety of $\fK$-algebras $\cV$ is such that for all algebras $X$, $Y$, $Z$ in $\cV$ the comparison map
    \[
        \Phi_{X,Y,Z} \colon X \diamond (Y \diamond Z) \to X \diamond Y \diamond Z
    \]
    is surjective, if and only if elements $\LLL_i$, $\LLR_i$, $\LRL_i$ and $\LRR_i$, $i\in \{1,2,3,4\}$ exist in $\fK$ for which
    \begin{equation}\label{Eq Lambda Rules}
        \begin{cases}
            (xy)z = \LLL_1 x(yz) + \LLL_2 x(zy) + \LLL_3 (yz)x + \LLL_4 (zy)x \\
            (yx)z = \LLR_1 x(yz) + \LLR_2 x(zy) + \LLR_3 (yz)x + \LLR_4 (zy)x \\
            z(xy) = \LRL_1 x(yz) + \LRL_2 x(zy) + \LRL_3 (yz)x + \LRL_4 (zy)x \\
            z(yx) = \LRR_1 x(yz) + \LRR_2 x(zy) + \LRR_3 (yz)x + \LRR_4 (zy)x
        \end{cases}
    \end{equation}
    are identities of $\cV$.\noproof
\end{proposition}

This immediately excludes certain homogeneous varieties:

\begin{examples}\label{Example Not Surjective}
    $\ALG$ does not satisfy the identities~\eqref{Eq Lambda Rules}, since it does not satisfy any non-trivial equations; more generally, any homogeneous variety of algebras that does not have any degree three identities cannot have surjective comparison maps $\Phi_{X,Y,Z}$. For instance, the variety of non-associative commutative algebras---the subvariety of $\ALG$ determined by the identity $xy-yx=0$, called \defn{commutative-magmatic} algebras in~\cite{Zinbiel}---is such.
\end{examples}

\begin{example}\label{Example NCA}
    The variety of (non-commutative) associative algebras does not satisfy the second identity in \eqref{Eq Lambda Rules}, so it cannot be cosmash associative---even though it satisfies a degree three equation.
\end{example}

Assuming (anti-)commutativity, what should we add for the $\Phi_{X,Y,Z}$ to become surjective? We may deduce from the first equation that some $\lambda\in \fK$ exists such that
\begin{align*}
    (xy)z = \lambda x(yz).
\end{align*}
In the commutative case, this implies that
\begin{align*}
    x(yz) = (yz)x = \lambda y(zx) = \lambda y(xz) = (yx)z = (xy)z = \lambda x(yz) \,,
\end{align*}
whence $\lambda=1$, meaning the algebras of our variety are associative. The same calculation as above in the anti-commutative case nets $x(yz) = - \lambda x(yz)$, which implies that $\lambda=-1$, meaning the algebras of our variety are anti-associative.

\begin{lemma}\label{Lemma Comm implies Assoc}
    Let $\cV$ be a homogeneous variety of $\fK$-algebras such that for all algebras $X$, $Y$, $Z$ in $\cV$ the comparison map
    \[
        \Phi_{X,Y,Z} \colon X \diamond (Y \diamond Z) \to X \diamond Y \diamond Z
    \]
    is surjective. If (anti-)commutativity holds in $\cV$ then so does (anti-)associativity.\noproof
\end{lemma}

\begin{example}\label{Example Lie fails surjectivity}
    By the above reasoning, Lie algebras are essentially excluded: indeed, since the validity of the equations makes anti-commutativity imply anti-associativity,  from the Jacobi identity we may deduce
    \begin{align*}
        0 & = x(yz)+y(zx)+z(xy) \ \ \    & \text{(Jacobi identity),}       \\
              & =x(yz)-(yz)x-(xy)z \ \ \  & \text{(anti-associativity and anti-commutativity),} \\
              & =x(yz)+x(yz)+x(yz) \ \ \   & \text{(anti-commutativity and anti-associativity).}
    \end{align*}
    When $\kar(\fK)\neq 3$, we conclude that $x(yz)=0$ and so, since not all Lie algebras are $2$-nilpotent, the variety of $\fK$-Lie algebras does not satisfy the equations \eqref{Eq Lambda Rules} of Proposition~\ref{Proposition Lambda Rules}. We come back to the equation $zxy=0$ in Example~\ref{Example 2-nilpotent magmatic algebras}, and eliminate the characteristic $3$ case in Proposition~\ref{Proposition Injectivity AC}.
\end{example}

We may now consider some positive examples.

\begin{example}\label{Example Surjectivity AC}
    Each of the varieties $\CA$ and $\ANTI$ is an example of an operadic variety of algebras over a field $\fK$ in which the equations \eqref{Eq Lambda Rules} of Proposition~\ref{Proposition Lambda Rules} can be satisfied. For instance, in the first case, we may take
    \[
        \LLL_1=\LLR_1=\LRL_1=\LRR_1=1
    \]
    and all other coefficients zero.
\end{example}

\begin{example}\label{Example non-classical surjective}
    We may create non-classical examples by means of a careful choice of coefficients in \eqref{Eq Lambda Rules}. For instance, the subvariety $\cV$ of $\ALG$ determined by
    \begin{equation*}
        \begin{cases}
            (xy)z = (yz)x = (zx)y \\
            x(yz) = y(zx) = z(xy)
        \end{cases}
    \end{equation*}
    fits those identities, because we may choose
    \[
        \LRL_1=\LRR_2=\LLL_3=\LLR_4=1
    \]
    and all other coefficients zero.

    Note that here, the comparison maps $\Phi_{X,Y,Z} \colon {X \diamond (Y \diamond Z) \to X \diamond Y \diamond Z}$ need not be injective: if we take $X=A$, $Y=B$ and $Z=C$ the free algebras is this subvariety generated by elements $a$, $b$ and $c$, we may calculate that
    \[
        (a(bc))a=(b(ca))a=((ca)a)b=((aa)c)b=(cb)(aa)
    \]
    in $A \diamond B \diamond C$. However, $(aa)(bc)=(a(bc))a=(cb)(aa)$ need not hold in $A \diamond (B \diamond C)$ by the reasoning in Remark~\ref{Remark Independence Variations}.
\end{example}

\begin{example}\label{Example perm algebras}
    Another, quite similar example is the operadic variety of so-called \defn{perm algebras} (\cite{Chapoton} and~\cite[page 235]{Zinbiel}), which are determined by the identities $(xy)z=x(yz)=x(zy)$. Here we may choose
    \[
        \LLL_1=\LLR_3=\LRL_4=\LRR_4=1
    \]
    and all other coefficients zero.
\end{example}

\begin{remark}\label{Remark AC}
    The identities \eqref{Proposition Lambda Rules} imply an instance of the equations
    \[
        \begin{cases}
            \begin{aligned}
                z(xy)=
                \lambda_{1}y(zx) & +\lambda_{2}x(yz)+
                \lambda_{3}y(xz)+\lambda_{4}x(zy)                      \\
                                 & +\lambda_{5}(zx)y+\lambda_{6}(yz)x+
                \lambda_{7}(xz)y+\lambda_{8}(zy)x
            \end{aligned} \\[3ex]
            \begin{aligned}
                (xy)z=
                \lambda_{9}y(zx) & +\lambda_{10}x(yz)+
                \lambda_{11}y(xz)+\lambda_{12}x(zy)                      \\
                                 & +\lambda_{13}(zx)y+\lambda_{14}(yz)x+
                \lambda_{15}(xz)y+\lambda_{16}(zy)x
            \end{aligned}
        \end{cases}
    \]
    ($\lambda_{1}$, \dots, $\lambda_{16}\in \fK$) which characterise~\cite[Theorem 2.12]{GM-VdL2} those varieties of $\fK$-algebras that are \emph{algebraically coherent} in the sense of~\cite{acc}. Indeed, we may take
    \[
        \lambda_{2k}= \LRL_k,\qquad \lambda_{2k+8}= \LLL_k
    \]
    for $1\leq k\leq 4$ and all other $\lambda_i$ zero. Where the identities \eqref{Proposition Lambda Rules} allow us to ``pull the $x$ out of the parentheses'', the identities characterising algebraic coherence allow us to ``pull the $z$ into the parentheses''; which is clearly weaker, since to pull out the $x$, we need to pull in the $z$. Hence any homogeneous variety of $\fK$-algebras with surjective comparison maps $\Phi_{X,Y,Z} \colon {X \diamond (Y \diamond Z) \to X \diamond Y \diamond Z}$ is an algebraically coherent category. As a consequence, any such category has the convenient categorical-algebraic properties mentioned in~\cite{acc}. In~\cite{RVV2}, we further investigate the cosmash associativity condition from a general categorical-algebraic perspective.

    On the other hand, the variety $\LIE$ is algebraically coherent, but we just saw that it does not satisfy the equations of Proposition~\ref{Proposition Lambda Rules}. The same holds for the variety of (non-commutative) associative algebras considered in Example~\ref{Example NCA}.
\end{remark}

\section{Injectivity of the comparison maps}\label{Section Injectivity}

\subsection{Examples and counterexamples}
It is easy to give examples of varieties of $\fK$-algebras whose comparison morphisms $\Phi_{X,Y,Z} \colon {X \diamond (Y \diamond Z) \to X \diamond Y \diamond Z}$ fail to be injective.

\begin{examples}\label{Examples Injectivity}
    \begin{enumerate}
        \item We may, for instance, consider the variety of $\fK$-algebras that satisfy no other identities besides the identity $x(yz)-x(zy)=0$: in the situation of Lemma~\ref{Lemma Independence}, an identity of degree two is missing for us to conclude that $a(bc)=a(cb)$ in $A \diamond (B \diamond C)$, even though this equality holds in~$A \diamond B \diamond C$.
        \item Example~\ref{Example non-classical surjective} shows that the comparison maps $\Phi_{X,Y,Z}$ need not be injective even when they are surjective.
    \end{enumerate}
\end{examples}

Lemma~\ref{Lemma Independence} provides us with an additional technique to exclude cosmash associativity for certain types of algebras. Indeed, when each morphism
\[
    \Psi_{X,Y,Z}\colon (X\diamond Y)\diamond Z\rightarrow X\diamond Y\diamond Z
\]
is an injection, if $A$, $B$ and~$C$ are free algebras as in the statement of the lemma, then
\[
    \lambda_1(ab)c+\lambda_2(ba)c+\lambda_3c(ab)+\lambda_4c(ba)=0
\]
in $A\cosmash B \cosmash C$ implies $\lambda_i=0$ for all $1\leq i\leq 4$. Hence, for instance:

\begin{example}\label{Example 2-nilpotent magmatic algebras}
    An extreme variation on the theme in Examples~\ref{Examples Injectivity} is the subvariety $\NIL$ of $\ALG$ determined by the equations
    \begin{equation*}
        (xy)z=0=x(yz).
    \end{equation*}
    Here all ternary cosmash products vanish, so that each of the comparison maps $\Phi_{X,Y,Z}$ is surjective. The above technique, however, allows us to prove that this variety is not cosmash associative. The reason is, that a repeated binary cosmash product in $\NIL$ may be non-trivial. For instance, if $A$, $B$ and $C$ are freely generated by $a$, $b$ and $c$ as above, then $(ab)c$ is a non-zero element of $(A\cosmash B)\cosmash C$ by Lemma~\ref{Lemma Independence}.
\end{example}

We note that the comparison morphisms $\Phi_{X,Y,Z}$ can be injective without being all surjective:

\begin{example}\label{Example ALG not injective}
    The operadic variety $\ALG$ itself is such. Indeed, by lack of any non-trivial identities in this variety, any equality between elements of $X \diamond (Y \diamond Z)$ must either be induced by an equality in $X$, $Y$ or $Z$ or by basic polynomial manipulations. None of this happens more easily in $X \diamond Y \diamond Z$, so that the map $\Phi_{X,Y,Z}$ is indeed injective. Its non-surjectivity was already commented upon in Examples~\ref{Example Not Surjective}.
\end{example}

\begin{remark}
    When $A$, $B$ and $C$ are free non-associative algebras, Example~\ref{Example ALG not injective} implies that the elements of $A \diamond (B \diamond C)$ are polynomials, since they may now be viewed as elements of the free algebra $A+B+C$. Part of this extends to the free algebras of any variety of $\fK$-algebras $\cV\leq\ALG$, even for  those varieties which are not cosmash injective: indeed, their cosmash product---whose elements are thus polynomials modulo identities in $\cV$---will be a quotient of the cosmash product $A \diamond (B \diamond C)$ in $\ALG$ for some free non-associative algebras $A$, $B$ and $C$. Here we may essentially follow the reasoning of Subsection~2.5 of~\cite{GM-VdL2}. This provides a complementary view on~\ref{Subsection Linear Independence}.
\end{remark}

\begin{example}
    Recall from Examples~\ref{Example Not Surjective} that also the operadic variety of non-associative commutative algebras has non-surjective $\Phi_{X,Y,Z}$. On the other hand, all $\Phi_{X,Y,Z}$ are injective, since any equality between two elements of the cosmash product $X \diamond Y \diamond Z$ is induced by a finite chain of applications of the commutative law. (For instance, if $x\in X$,~$y\in Y$ and $z\in Z$, then $x(yz)+x((yz)z)=x(zy)+x((yz)z)$ is such.) Each such application, however, is also valid in the algebra $X \diamond (Y \diamond Z)$, as an instance of commutativity either in $X + (Y \diamond Z)$ or in $Y+Z$. It follows that there are no equalities in $X \diamond Y \diamond Z$ between elements of $X \diamond (Y \diamond Z)$ which do not already hold in $X \diamond (Y \diamond Z)$.
\end{example}

The case of commutative algebras is in some sense generic. In what follows, we prove that in a variety with surjective $\Phi_{X,Y,Z}$, these comparison maps can only be injective if an identity of degree two holds (Proposition~\ref{Proposition Degree Two}). By means of Lemma~\ref{Lemma Degree Two}, we may then deduce that the algebras are either commutative or anti-commutative.

\subsection{Preparatory results}
We first need to extend Lemma~\ref{Lemma Independence} to a result which holds for quaternary cosmash products. In a category $\cC$ in which all comparison maps $\Phi_{X,Y,Z}$ are monomorphisms, we consider objects $X$, $Y$, $Z$ and $W$. Then the natural transformation
\[
    \Phi_{-,Z,W}\colon (-)\cosmash (Z\cosmash W) \to (-)\cosmash Z\cosmash W\colon \cC\to \cC
\]
is a monomorphism. On the other hand, by the second item in Examples~\ref{Examples Cross-Effects},
\[
    X\cosmash Y\cosmash (Z\cosmash W)\cong \cre_2((-)\cosmash (Z\cosmash W))(X,Y),
\]
while by the third item in Examples~\ref{Examples Cross-Effects},
\[
    X\cosmash Y\cosmash Z\cosmash W\cong \cre_2((-)\cosmash Z\cosmash W)(X,Y).
\]
Since by Remark~\ref{Remark Functor cr}, the functor $\cre_2$ preserves natural monomorphisms, we thus find a monomorphism
\begin{align*}
    X\cosmash Y\cosmash (Z\cosmash W) & \cong \cre_2((-)\cosmash (Z\cosmash W))(X,Y) \\
                                      & \to \cre_2((-)\cosmash Z\cosmash W)(X,Y)     \\
                                      & \cong X\cosmash Y\cosmash Z\cosmash W,
\end{align*}
which when composed with the injection
\[
    \Phi_{X,Y,Z\cosmash W}\colon(X\cosmash Y)\cosmash (Z\cosmash W)\to X\cosmash Y\cosmash (Z\cosmash W)
\]
yields an inclusion $(X\cosmash Y)\cosmash (Z\cosmash W)\to X\cosmash Y\cosmash Z\cosmash W$ which happens to be the canonical map. Hence:

\begin{lemma}\label{Lemma Comparison Injective Implication}
    The canonical comparison morphism
    \[
        (X\cosmash Y)\cosmash (Z\cosmash W)\to X\cosmash Y\cosmash Z\cosmash W
    \]
    is a monomorphism whenever $\Phi_{-,Z,W}\colon (-)\cosmash (Z\cosmash W) \to (-)\cosmash Z\cosmash W\colon \cC\to \cC$ is a natural monomorphism. \noproof
\end{lemma}

\begin{lemma}\label{Lemma Independence Ternary}
    Suppose $\cV$ is a variety of $\fK$-algebras which satisfies no non-trivial identities of degree smaller than or equal to $2$ and in which the comparison map
    \[
        \Phi_{X,Y,Z} \colon {X \diamond (Y \diamond Z) \to X \diamond Y \diamond Z}
    \]
    is always injective. Let $A$, $B$, $C$ and $D$ be free algebras in $\cV$ on a single generator written $a$, $b$, $c$ and $d$, respectively. Then the sets of polynomials
    \[
        \{(ab)c, (ba)c, c(ab), c(ba)\}
    \]
    and
    \[
        \{(ab)(cd), (ba)(cd), (ab)(dc), (ba)(dc), (cd)(ab), (cd)(ba), (dc)(ab), (dc)(ba)\}
    \]
    are linearly independent in the free algebra $A+B+C+D$.

    Furthermore, if either commutativity or anti-commutativity is added to the identities of~$\cV$, then $(ab)c$ and $(ab)(cd)$ are still non-zero in $A+B+C+D$.
\end{lemma}
\begin{proof}
    Independence of the first set follows from Lemma~\ref{Lemma Independence} and the comment immediately following it, once we notice that the coproduct $A+B+C$ is the free algebra on $\{a,b,c\}$. For the second set, we note that by Lemma~\ref{Lemma Comparison Injective Implication}, the comparison map
    \[
        (A \cosmash B) \cosmash (C \cosmash D) \to A \cosmash B \cosmash C \cosmash D
    \]
    is injective as well. Analogously to the reasoning in Lemma~\ref{Lemma Independence}, we can prove the given polynomials are indeed linearly independent in ${A+B+C+D}$. The final statement follows from a straightforward variation on the reasoning leading up to Lemma~\ref{Lemma Independence} in the settings of commutative and anti-commutative algebras, where the structure of the coproduct is simpler.
\end{proof}

We may now return to Example~\ref{Example Surjectivity AC} where we explained that the variety $\ANTI$ of anti-commutative anti-associative $\fK$-algebras has surjective $\Phi_{X,Y,Z}$. It is not cosmash associative, though---unless $\kar(\fK)=2$, which is just the commutative case again:

\begin{proposition}\label{Proposition Injectivity AC}
    If $\fK$ is a field of characteristic different from $2$, then a variety of anti-commutative anti-associative $\fK$-algebras is only cosmash associative when it is abelian.
\end{proposition}
\begin{proof}
    This follows from the well-known fact that we may deduce the identity $xyzt=0$ from anti-associativity. Indeed,
    \[
        (xy)(zt)=-((xy)z)t=(x(yz))t=-x((yz)t)=x(y(zt))=-(xy)(zt),
    \]
    whence $(xy)(zt)=0$ follows---as long as $\kar(\fK)\neq 2$. Letting $A$, $B$, $C$, $D$ be free with respective generators $a$, $b$, $c$, $d$ we have that, in particular, the element $(ab)(cd)$ of $A+B+C+D$ is zero. Lemma~\ref{Lemma Independence Ternary} now tells us that, if  the variety in question is non-abelian, then the canonical $\Phi_{X,Y,Z}$ cannot all be injective.
\end{proof}

\subsection{Result and proof}
From this, combined with Lemma~\ref{Lemma Comm implies Assoc} and Lemma~\ref{Lemma Degree Two}, we may now deduce that amongst varieties of algebras which satisfy an identity of degree two, only the commutative associative ones are potentially cosmash associative. In what follows, our strategy is to show that injectivity of the comparison maps, together with the identities \eqref{Eq Lambda Rules}, implies the existence of such a degree two identity. Indeed, assuming a variety to be cosmash associative entails some linear dependence conditions (Proposition~\ref{Proposition Lambda Rules}), while additionally assuming the non-existence of non-trivial degree $\leq 2$ identities implies some linear \emph{in}dependence conditions (Lemma~\ref{Lemma Independence Ternary}). We will now proceed explain how one can show that these two assumptions contradict each other, allowing us to obtain Proposition~\ref{Proposition Degree Two}. This mimics the approach followed in the articles~\cite{GM-VdL2,GM-VdL3}: the former uses elementary methods to arrive at a characterisation (of a different property, irrelevant for our present purposes) in the presence of a degree two identity, whereas in the latter a proof by computer is needed to show that such a degree two identity must hold.

\begin{proposition}\label{Proposition Degree Two}
    If a homogeneous variety of $\fK$-algebras $\cV$ has an associative cosmash product, then the variety must satisfy at least one non-trivial identity of degree $\leq 2$.
\end{proposition}

Suppose we are given any multilinear monomial $M$ in $n \geq 3$ variables $x_1$, \dots, $x_n$ and any multilinear monomial $X$ in a subset of the variables $\{x_1 , \ldots, x_n \}$. Observe that there is at most one identity among \eqref{Eq Lambda Rules} that can be used to pull $X$ one step out of the parentheses.
This allows us to define a linear map $H^X$ from multilinear polynomials (in the variables $x_i$) to multilinear polynomials. For each multilinear monomial $M$ we define $H^X(M)$ as follows:
\begin{enumerate}
    \item if one of the left hand sides of the identities \eqref{Eq Lambda Rules}, where $x$ is taken to be $X$, is applicable to the monomial $M$, then we let $H^X(M)$ be the polynomial we get by substituting the corresponding right hand side of the applicable identity into $M$;
    \item otherwise, we let $H^X(M)$ equal $M$.
\end{enumerate}
By linear extension we can apply the map $H^X$ to any multilinear polynomial. We note that the case (2) happens precisely when $X$ does not appear as a submonomial of $M$, or when $M=XY$ or $M=YX$ for some submonomial $Y$.

In effect, $H^X$ pulls the submonomial $X$ one step outside of the parentheses in each monomial, whenever possible. For example, if $p(a,b,c,d) = d(a(bc))$, then $H^b(p)$ will be
\begin{equation*}
    \LRL_1 d(b(ca)) + \LRL_2 d(b(ac)) + \LRL_3 d((ca)b) + \LRL_4 d((ac)b) \,.
\end{equation*}
On the other hand
\[
    H^b(b(a(dc))) = b(a(dc))\,,
\]
because none of the left hand sides of the identities \eqref{Eq Lambda Rules} is applicable when $x = b$, since this $b$ is already completely outside of the parentheses.

Note that $H^X$ acts by transforming non-associative polynomials using the identities \eqref{Eq Lambda Rules} that hold in any cosmash associative variety. Therefore, when viewing polynomials modulo the identities of a cosmash associative variety, $H^X$ behaves as the identity transformation.

We can now proceed as follows. Starting with a polynomial such as $p(a,b,c) = (ab)c$, we have
\begin{align*}
    H^b(p) & = \LLR_1 b(ac) + \LLR_2 b(ca) + \LLR_3 (ac)b + \LLR_4 (ca)b
\end{align*}
and
\begin{align*}
    H^b(H^a(p)) = \, & ( \LLL_2 \LLL_3+ \LLL_4 \LLR_2+ \LLL_1 \LRL_2 + \LLL_2 \LRR_2)b(ac)    \\
                     & + ( \LLL_1 \LLL_3+ \LLL_4 \LLR_1+ \LLL_1 \LRL_1+ \LLL_2 \LRR_1)b(ca)   \\
                     & + ( \LLL_3 \LLL_4 + \LLL_4 \LLR_4+ \LLL_1 \LRL_4+ \LLL_2 \LRR_4)(ac)b  \\
                     & + ( \LLL_3 \LLL_3 + \LLL_4 \LLR_3+ \LLL_1 \LRL_3+ \LLL_2 \LRR_3)(ca)b.
\end{align*}
We must have
\[
    H^b(H^a(p)) - H^b(p) = 0
\]
modulo the identities of $\cV$, since $\cV$ is assumed to be cosmash associative. Moreover, observe that the monomials of $H^b(H^a(p)) - H^b(p)$ are scalar multiples of the elements of the set
\[
    \{ b(ca),b(ac), (ca)b , (ac)b \}
\]
of linearly independent polynomials (Lemma~\ref{Lemma Independence Ternary}) in the free algebra in $\cV$ generated by $\{a,b,c\}$. Therefore, the coefficients of the monomials must be zero, giving us the four equations
\begin{equation*}
    \begin{cases}
        \LLL_2\LLL_3 + \LLL_4\LLR_2+ \LLL_1\LRL_2+ \LLL_2\LRR_2 - \LLR_1 = 0 \\
        \LLL_1\LLL_3+ \LLL_4\LLR_1+ \LLL_1\LRL_1 +\LLL_2\LRR_1 - \LLR_2 = 0  \\
        \LLL_3\LLL_4 +\LLL_4\LLR_4 +\LLL_1\LRL_4 + \LLL_2\LRR_4 - \LLR_3 = 0 \\
        \LLL_3 \LLL_3 +\LLL_4\LLR_3 +\LLL_1\LRL_3 +\LLL_2\LRR_3 - \LLR_4 = 0 \,.
    \end{cases}
\end{equation*}
Each of these equations expresses an equality involving elements $\LLL_i$, $\LLR_i$, $\LRL_i$, $\LRR_i$ of the field $\fK$, which allows us to view their left hand sides as \emph{associative commutative} polynomials over $\Z$ with variables in $\fK$. Hence these four equations may be considered as equations in the ``ordinary'' (= associative commutative unitary) polynomial algebra over the symbols $\LLL_i$, $\LLR_i$, $\LRL_i$, $\LRR_i$, $1\leq i\leq 4$, which we shall here denote $\Z\ldbrack\LLL_i, \LLR_i, \LRL_i,\LRR_i\rdbrack$.

Our end goal is to reach an inconsistent system of equations by varying $p$ and the applications of $H^X$, with each choice giving us a set of equations in the variables $\LLL_i$, $\LLR_i$, $\LRL_i$ and $\LRR_i$, which need to be satisfied under our assumptions.

We get a total of $32$ degree $2$ equations by applying $H^X$ in various ways to $p = (ab)c$ and $p = a(bc)$. However, these $32$ equations do not form an inconsistent system, so we add another $64$ degree $3$ equations, which we produce using the same schema as described above, except $p$ will be a degree $4$ monomial and we will be using the independence of the set
\[
    \{(ab)(cd), (ba)(cd), (ab)(dc), (ba)(dc), (cd)(ab), (cd)(ba), (dc)(ab), (dc)(ba)\}
\]
of monomials, which is stated in Lemma~\ref{Lemma Independence Ternary}. The full set of equations is listed in Appendix~\ref{Appendix Equations}, along with the corresponding choices of $p$ and the applications of~$H^X$ which generate them.
We note that there is a lot of redundancy in that set of equations. By random sampling we have found a subset of 39 equations which is inconsistent as well. However, it is computationally less intensive to prove the inconsistency of the full set of equations.

\begin{proof}[Proof of Proposition~\ref{Proposition Degree Two}]
    Using Lemma~\ref{Lemma Independence Ternary} repeatedly on several polynomials~$p$, the above process yields a system of equations $(f_i=0)_{1\leq i\leq 96}$ in the algebra  $\Z\ldbrack\LLL_i, \LLR_i, \LRL_i,\LRR_i\rdbrack$; see Appendix~\ref{Appendix Equations}. We show that this system is inconsistent, by providing coefficients $\mu_i$ in $\Z\ldbrack\LLL_i, \LLR_i, \LRL_i,\LRR_i\rdbrack$ such that $\sum\mu_i f_i$ is a non-zero integer in~$\fK$. Clearly then, a common solution for the equations $(f_i=0)_{1\leq i\leq 96}$ cannot exist. Hence, the hypothesis of Lemma~\ref{Lemma Independence Ternary} that $\cV$ does not satisfy any non-trivial equations of degree~$2$ must be false.

    The size of the system makes it impossible to do the calculations by hand. We used the open-source software package \Singular~\cite{DGPS} both for generating the system of equations and proving its inconsistency (by means of a Gröbner basis calculation). The full code is available as a set of ancillary files to the arXiv version of this article, while the code as well as its output are accessible via~\cite{RVV-Code}.

    We first find coefficients $\mu_i$ such that $\sum\mu_i f_i$ equals the integer $m$ in Appendix~\ref{Appendix m m'}. This already shows the inconsistency of the system when either $\fK$ has characteristic zero, or $\fK$ has a prime characteristic~$p$ which does not divide $m$. We exclude all the other prime characteristics as follows. We first redo the calculations with the roles of $\delta_3$ and $\delta_4$ swapped, which yields coefficients $\nu_i$ such that $\sum\nu_i f_i$ is the integer $m'$ in Appendix~\ref{Appendix m m'}. The Euclidean algorithm tells us that the greatest common divisor of $m$ and $m'$ is~$2$. Hence we only need to check that the system is inconsistent when $\fK$ has characteristic $2$, and a separate calculation shows that this is indeed the case.
\end{proof}

\begin{remark}
    We checked this by means of an independent calculation in the software package \Mathematica~\cite{Mathematica}. The code in~\cite{RVV-Code} generates output---note that the file containing the $f_i$ and $\mu_i$ is round $130\,\mathrm{MB}$ large---which may be used to check in \Mathematica\ that indeed $\sum\mu_i f_i=m$ and  $\sum\nu_i f_i=m'$ for the coefficients generated by \Singular. We preferred to explain how to arrive at a solution using \Singular\ because it is a freely available open source package, which for the present task seems at least as efficient as \Mathematica.
\end{remark}

Using Lemma~\ref{Lemma Degree Two} and Proposition~\ref{Proposition Injectivity AC}, from this we may deduce:

\begin{theorem}\label{deg-2-identities}
    If a homogeneous variety of $\fK$-algebras is cosmash associative, then it must satisfy at least one of the identities $xy - yx$, $xy$ or $x$. As a consequence, it is a subvariety of $\CA$. \noproof
\end{theorem}

\begin{remark}\label{Remark integral domain}
    The reasoning leading up to Proposition~\ref{Proposition Degree Two} stays valid when $\fK$ is an integral domain. On the other hand, its interpretation in Theorem~\ref{deg-2-identities} makes use of Lemma~\ref{Lemma Degree Two}, which is not valid at that level of generality. It seems the best we can do is consider varieties that satisfy a law of the form $\lambda xy+\mu yx=0$, something we shall not pursue in this article.
\end{remark}

\section{The operadic case}\label{Section Operadic Case}

We are now going to show that in the operadic context (where, recall, the varieties may be presented in terms of multilinear identities), the only cosmash associative subvarieties of $\CA$ are $\ABALG$ and $\TRIV$. By Remark~\ref{Remark operadic}, the following then applies in particular when the characteristic of $\fK$ is zero.

\begin{proposition}\label{Proposition Subvarieties}
    Let $\cV$ be an operadic variety of $\fK$-algebras with the cosmash associativity property. If $\cV$ is a proper subvariety of $\CA$, then either $\cV = \ABALG$ or $\cV =\TRIV$.
\end{proposition}
\begin{proof}
    We consider a proper subvariety $\cV$ of $\CA$ whose cosmash product is associative. Since it is a proper subvariety, there exists some homogeneous polynomial $\varphi$ characterising $\cV$ which cannot be deduced from associativity and commutativity. We prove by induction on the degree of $\varphi$ that its existence combined with the cosmash associativity property implies that either $xy=0$ or $x=0$ is an identity of~$\cV$.

    First, if $\deg(\varphi)=1$, then $x=0$ trivially holds in $\cV$.

    Next, if $\deg(\varphi)=2$, then Lemma~\ref{Lemma Degree Two} tells us that either $xy+yx=0$ or $xy=0$ holds in~$\cV$. If the characteristic of $\fK$ is $2$, then $xy+yx=0$ is an identity of $\CA$, so $xy=0$ must hold in the proper subvariety $\cV$. Otherwise, the weakest non-trivial equation of degree two we may add to commutativity is anti-commutativity. Then both commutativity and anti-commutativity are identities of the variety $\cV$, which implies that $xy=0$ holds.

    Finally, if $\deg(\varphi)=n \geq 3$, then by using commutativity and associativity the identity can be rewritten as
    \begin{equation}\label{Non-multilinear Equation}
        (x_1 \cdots x_{n-2})(x_{n-1}x_n)=0
    \end{equation}
    where by multilinearity all of the $x_i$ are different. Then, by considering $X$, $Y$ and~$Z$ to be the free algebras in $\cV$ respectively generated by $\lbrace x_1, \dots , x_{n-2}\rbrace$, $\lbrace x_{n-1}\rbrace$ and $\lbrace x_n \rbrace$, the hypothesis that the comparison map
    \begin{align*}
        X \diamond (Y\diamond Z) \to X \diamond Y \diamond Z
    \end{align*}
    is an isomorphism implies that $(x_1 \cdots x_{n-2})(x_{n-1}x_n)=0$ in the algebra
    \[
        {X\diamond (Y\diamond Z)}\leq {X+(Y\diamond Z)}.
    \]
    Lemma~\ref{Lemma Independence Ternary} allows us to vary on the reasoning in Remark~\ref{Remark Independence Variations} in order to deduce that either the identity $xy=0$ holds in $\cV$, or $x_{n-1}x_n=0$ in $Y\diamond Z$ so that $xy=0$ holds in $\cV$, or $x_1 \cdots x_{n-2} =0$ in~$X$, which gives a simple induction argument leading to the same conclusion.
\end{proof}

\begin{theorem}\label{main-result}
    Let $\fK$ be a field. The cosmash product in an operadic variety of $\fK$-algebras $\cV$ is associative if and only if $\cV$ is one of the following:
    \begin{enumerate}
        \item \label{main-result:comm} the variety $\CA$ of associative and commutative $\fK$-algebras;
        \item \label{main-result:abelian} the variety $\ABALG$ of abelian $\fK$-algebras, which is isomorphic to $\VECT$;
        \item \label{main-result:zero} the variety $\TRIV$, containing only $\fK$-algebras of cardinality one.
    \end{enumerate}
\end{theorem}
\begin{proof}
    That cosmash associativity holds for these varieties is clear from Section~\ref{Section Cosmash Is Tensor} and the fact that additive categories have trivial cosmash products (Examples~\ref{examples and counterexamples of cosmash asso}). The other implication follows from Theorem~\ref{deg-2-identities} and Proposition~\ref{Proposition Subvarieties}.
\end{proof}

\section{Non-operadic varieties}\label{Section Non-Operadic Case}

Certain aspects of the theory may be extended to non-operadic varieties of $\fK$-algebras, which gives rise to some new examples, as well as a few complications.

In this context, what we call a \defn{variety of $\fK$-algebras} as in Definition~\ref{Definition identity variety} is simply a subvariety of~$\ALG$ in the sense of universal algebra: a~collection of $\fK$-algebras that satisfy a set of (not necessarily homogeneous or multilinear, but necessarily polynomial) equations. For instance:

\begin{example}[Alternating (anti-)associative algebras]\label{Example Alt}
    An algebra is \defn{alternating} if the identity $xx=0$ holds. We write $\ALT$ for the variety of anti-associative and alternating algebras over $\fK$. The variety $\ALT$ only differs from $\ANTI$ when $\kar(\fK)=2$; then the former is strictly smaller, because anti-commutativity $xy=-yx$ does not imply $xx=0$ in this case (cf.\ Remark~\ref{Remark operadic}). Rather, then the algebras are associative and commutative, so $\ALT$ may be seen as the subvariety of $\CA$ determined by the identity $xx=0$. Still under the condition that $\kar(\fK)=2$, it is easy to check that the coproduct of two alternating (anti-)associative $\fK$-algebras in the variety $\CA$ does actually lie in $\ALT$, which makes it the coproduct there as well. Hence the proof of Proposition~\ref{Proposition Cosmash is Tensor} applies, and the variety $\ALT$ has an associative cosmash product, just like $\CA$.
\end{example}

When $\fK$ is an infinite field, any variety of $\fK$-algebras is determined by its homogeneous identities only, thanks to the following result.

\begin{lemma}[\cite{ZSSS}]\label{identity-lemma}
    If $\cV$ is a variety of algebras over an infinite field, then
    all of its identities are of the form $\phi(x_1,\ldots, x_n) = 0$, where $\phi(x_1,\ldots, x_n)$ is a polynomial. Moreover, each homogeneous component $\psi(x_{i_1},\ldots, x_{i_m})$ of such an identity again gives rise to an identity $\psi(x_{i_1},\ldots, x_{i_m})=0$.\noproof
\end{lemma}

Hence, while such a variety need not be operadic, the results of Section~\ref{Section Surjectivity} and Section~\ref{Section Injectivity} stay true. Thus we obtain the following:

\begin{proposition}
    Over an infinite field $\fK$, any variety of $\fK$-algebras with associative co\-smash products is a variety of associative and commutative $\fK$-algebras.\noproof
\end{proposition}

On the other hand, Proposition~\ref{Proposition Subvarieties} becomes more complicated if the varieties are not operadic, because multilinearity of the equations plays a crucial role here. In characteristic zero, Theorem~\ref{main-result} stays valid as is because of Remark~\ref{Remark operadic}. In prime characteristics this is no longer true, since the multilinearisation process may fail. It is exactly this which is exploited in the following example, which generalises Example~\ref{Example Alt}.

\begin{example}[$x^p=0$ in characteristic $p$]
    The variety $\CAp$ from Examples~\ref{Examples varieties} has an associative cosmash product. To see this, we simply add to the proof of Proposition~\ref{Proposition Cosmash is Tensor} a verification that the co\-product of two algebras that satisfy $x^p=0$ does again satisfy this equation. Indeed, if $x=(a,a'\otimes b',b)$ where $a^p=(a')^p=(b')^p=b^p=0$, then
    \begin{align*}
        x^p & =\bigl((a,0,b)+(0,a'\otimes b',0)\bigr)^p =(a,0,b)^p+(0,a'\otimes b',0)^p \\
            & =\bigl((a,0,0)+(0,0,b)\bigr)^p+(0,(a')^p\otimes (b')^p,0)                 \\
            & =(a,0,0)^p+(0,0,b)^p =(a^p,0,0)+(0,0,b^p)=0
    \end{align*}
    by the binomial theorem (used twice) and the fact that since $\kar(\fK)=p$, we have $\binom{p}{k}=0$ for each integer $0<k<p$.
\end{example}

Variations on this idea give rise to further examples; our aim is to study those in future work.

When the field $\fK$ is finite, Lemma~\ref{identity-lemma} no longer applies, and the identities of a variety of $\fK$-algebras may be non-homogeneous. This yields yet another kind of example, as for instance:

\begin{example}[Boolean rings]\label{Example Bool}
    The variety of Boolean rings may be seen as the subvariety $\BOOL_{\fK}$ of $\CA$, where $\fK$ is the finite field $\Z_2$, determined by the equation $x=xx$. It is easy to see that the coproduct of two $\Z_2$-algebras that satisfy $x=xx$ does again satisfy this equation. Hence we may extend Proposition~\ref{Proposition Cosmash is Tensor} to see that $\BOOL_{\Z_2}$ has an associative cosmash product. This shows that over finite fields, examples of a very different nature exist.
\end{example}

\section{Final remarks}\label{Section Final Remarks}
This work gives rise to several questions of quite diverse nature; here we state a few, often of a more categorical flavour, which we find particularly interesting. For some of these questions, we already know an answer---we plan to treat those in subsequent work.

\subsection{The lax-monoidal structure of the cosmash product}
The cosmash product, together with all of its higher-order versions and the canonical morphisms between them, forms a so-called \emph{(symmetric) lax-monoidal} structure on the category where it is considered. This point of view helps proving that cosmash associativity induces isomorphisms between all of the higher-order cosmash products as well, so that those become independent of any chosen bracketing. This development is the subject of the article~\cite{RVV3}; in the current article, Lemma~\ref{Lemma Comparison Injective Implication} gives one example of where this may lead.

\subsection{Categorical-algebraic consequences}
The examples---item (4) in~\ref{examples and counterexamples of cosmash asso}, for instance---indicate a close connection between properties of the cosmash product and properties of the induced (Higgins) commutators. Some general results on this relationship, and on its implications for Categorical Algebra, form the subject of the article~\cite{RVV2}: in our opinion, treating these in the present paper would make it less focused and too long. For instance, Remark~\ref{Remark AC} indicates a close relationship between cosmash associativity and algebraic coherence, which may well be valid outside the context of varieties of algebras over a field. A similar question makes sense for \emph{action accessibility}~\cite{BJ07}, which is known to be equivalent to algebraic coherence in our present context~\cite{Edinburgh}. The fact that there are no non-abelian varieties over a field which are both cosmash associative and \emph{locally algebraically cartesian closed}~\cite{Gray2012}---combine the start of the Introduction, Example~\ref{Example Lie fails surjectivity} and Example~\ref{Example 2-nilpotent magmatic algebras}---may also be an instance of a more general result.

\subsection{Further questions}
In the non-operadic case, we do not yet have a complete characterisation of the cosmash associative varieties of algebras. On the other hand, algebras over a field are very special, so are there any other examples? Remark~\ref{Remark integral domain} indicates that beyond the context of fields, certain results are still valid. And, what about cosmash associative varieties of groups, for instance? George Janelidze told us that the variety of \emph{differential graded rings} is cosmash associative; does this example fit a general pattern?

Finally: cosmash products are intrinsically symmetric. Can we characterise associativity for non-commutative algebras by means of some kind of non-symmetric cosmash product?

\section*{Acknowledgements}
We would like to thank Xabier García-Martínez for fruitful discussions on the subject of this article. The first author thanks the Université catholique de Louvain and IRMP for their hospitality during the writing of this paper.


\appendix

\section{The numbers \texorpdfstring{$m$ and $m'$}{m and m'} of Proposition~\ref{Proposition Degree Two}}\label{Appendix m m'}

\tiny
The number $m$ in the proof of Proposition~\ref{Proposition Degree Two} is
\begin{quote}
    $\seqsplit{74396795444726952714681837629686757881240662022215924745210907666577397441180147163464964221900491415624948921978905681863659912718475954729372154203903495273446853538276064959494310131192923769051184068535161006816210825040029629652997083427316077639698501360759914451399225447570785315441084997075121027375759347999077149538795566769957280247462901986878732962516216436338152096382505273589978994791621563274606732737945746221690145231014799621394065415995932666458922058579217942226350466610839335456731888664582265765584498632432615549674152508444592622255935921195596969138032872310751900552173274134607403718968536596033515889779681156951750418126790143352760027186609982149155274103331634448603177626071958991988431774390356744728517793552810949326985652930143291491349664287716617712953696517340386909696953499765267080568005594243345441370058305603316818161684991218210791806966661680528479796217605783513505959701955589499260346086801889033979888580872303550504010767286887701652647581383693253438897969930850140789169375700626730459389966931296382767714360170478838805480996110427502094998667818577669475228616493596843828856563398237405144244812678795323397772300506419759137217624265667188795105342969041752706959092923633786009598931405855596268136000169065316850401454162837723057594761011657683167811555169305845478478806338306621372765239069477161257910301103888314021610653760031920194648142391380660583196682184386769493167089282529043629317198851506306284426212763456951887877187260988027641352098896702838192128946348130910774022714695971415923138967341446570648449757966159632199432273509871727991402051661857055760183956912876231129472828158271237822875437178838933046838032875203107522465557155653974715245370419726147370009288477628184425570296255366586326873824993326507571208343702072702579137551525251738720048618757937148735558614770726433635794730485161672061413591959782054391246466720392229385803742447108575961668225764282550823260539590024775392568545963916466517954726718460531890085138462095009488493948319186544995507854452665364993006492371437112569937771459992368124388982487341083103876541230121660243641688648603635564275039223714342382341643239220080271381943138781673436155038}$,
\end{quote}
while $m'$ is equal to
\begin{quote}
    $\seqsplit{742839700513979359048299479947382299107179496017218434871605363438333569959566342563839977312088643747054988893738019193236830580350686697704189429387061790603105724165250446216101291211098709376230735924859262016628173855406613911879074744641561171079355282338173412879675253790342238192386393400270802735291525123094898077539503138451061904344486666671871626766269761074146844381799567124846479024535435268086110933813107573888607864833423265454216619525880119277995859545036670070827165075455687568866871395544565740411464520423207036587611398164012060384502798748322037005855583137142465733350495790997322093866926250976129000158920729052620739336368578744711102747141484155288077756995869958683709703473235719618842690521843663108018169921620109525786184803578543796393130465633886083978477919585756624340900521088853836736152465133190339304276798709673539506508004842183239052177431077654174239334020952015880060468802267001172105987015333729545825122330795138775202607270752857990665923542321887330695298254641336712116181999315214337739451097749804044079919638076327335571296538370152474826217092566696969386749267533232992758578455161805055475622775278399542859314452433776825056913365194716450492173099298557586864987290031096561495559088818397725499743514777552092423843493186947255888408168425803524249497298355907848117592572471307486966237062513887869473133646017396404024722734458163677562429482676808186923876755406242340095904091305673052377871685776998380347016787947937216838349386016751661418951834077834220442669244897298075130595074703242399313977942988374772287341330851237240414248477722647463141604230367901750465724712121880375680065290760717409362817386439052340226775186292960287608535212456229658172598999937503624729206343437423468263621205097544810494533672867412331182143200935652703516729404507531913977377030572803240346207388532660474898916510273808857946718363313130947612583075369197688162957609012191066751058610310544685082930165495225487207316887069998229752160192082481143424149584998788881376773544039278684005413389648004523874085957975788016819512928}$.
\end{quote}
\normalsize

\section{The system of equations \texorpdfstring{$(f_i=0)_{1\leq i\leq 96}$}{fi} of Proposition~\ref{Proposition Degree Two}} \label{Appendix Equations}

\tiny
\begin{multicols}{2}
    \ttab{1}{4}{p = (ab)c}{H^{c}(H^{a}(p)) - p = 0}
    {
    f_{1} &=  {\LLL_1}{\LLL_4}+{\LLL_3}{\LLR_1}+{\LLL_2}{\LRL_1}+{\LLL_1}{\LRR_1}   \\
    f_{2} &= {\LLL_2}{\LLL_4}+{\LLL_3}{\LLR_2}+{\LLL_2}{\LRL_2}+{\LLL_1}{\LRR_2}   \\
    f_{3} &= {\LLL_3}{\LLL_4}+{\LLL_3}{\LLR_3}+{\LLL_2}{\LRL_3}+{\LLL_1}{\LRR_3}   \\
    f_{4} &= {\LLL_4}{\LLL_4}+{\LLL_3}{\LLR_4}+{\LLL_2}{\LRL_4}+{\LLL_1}{\LRR_4}-1     }

    \ttab{5}{8}{p = (ab)c}{H^{c}(H^{b}(p)) - p = 0}
    {
    f_{5} &=  {\LLR_1}{\LLR_3}+{\LLL_1}{\LLR_4}+{\LLR_2}{\LRL_1}+{\LLR_1}{\LRR_1}   \\
    f_{6} &={\LLR_2}{\LLR_3}+{\LLL_2}{\LLR_4}+{\LLR_2}{\LRL_2}+{\LLR_1}{\LRR_2}   \\
    f_{7} &={\LLR_3}{\LLR_3}+{\LLL_3}{\LLR_4}+{\LLR_2}{\LRL_3}+{\LLR_1}{\LRR_3}-1 \\
    f_{8} &={\LLL_4}{\LLR_4}+{\LLR_3}{\LLR_4}+{\LLR_2}{\LRL_4}+{\LLR_1}{\LRR_4}       }

    \ttab{9}{12}{p = (ab)c}{H^{b}(H^{a}(p)) - H^{b}(p) = 0}
    {
    f_{9} &=  {\LLL_1}{\LLL_3}+{\LLL_4}{\LLR_1}+{\LLL_1}{\LRL_1}+{\LLL_2}{\LRR_1}-{\LLR_2} \\
    f_{10} &={\LLL_2}{\LLL_3}+{\LLL_4}{\LLR_2}+{\LLL_1}{\LRL_2}+{\LLL_2}{\LRR_2}-{\LLR_1} \\
    f_{11} &={\LLL_3}{\LLL_3}+{\LLL_4}{\LLR_3}+{\LLL_1}{\LRL_3}+{\LLL_2}{\LRR_3}-{\LLR_4} \\
    f_{12} &={\LLL_3}{\LLL_4}+{\LLL_4}{\LLR_4}+{\LLL_1}{\LRL_4}+{\LLL_2}{\LRR_4}-{\LLR_3}     }

    \ttab{13}{16}{p = (ab)c}{H^{a}(H^{b}(p)) - H^{a}(p) = 0}
    {
    f_{13} &=  {\LLL_1}{\LLR_3}+{\LLR_1}{\LLR_4}+{\LLR_1}{\LRL_1}+{\LLR_2}{\LRR_1}-{\LLL_2} \\
    f_{14} &={\LLL_2}{\LLR_3}+{\LLR_2}{\LLR_4}+{\LLR_1}{\LRL_2}+{\LLR_2}{\LRR_2}-{\LLL_1} \\
    f_{15} &={\LLL_3}{\LLR_3}+{\LLR_3}{\LLR_4}+{\LLR_1}{\LRL_3}+{\LLR_2}{\LRR_3}-{\LLL_4} \\
    f_{16} &={\LLL_4}{\LLR_3}+{\LLR_4}{\LLR_4}+{\LLR_1}{\LRL_4}+{\LLR_2}{\LRR_4}-{\LLL_3}     }

    \ttab{17}{20}{p = a(bc)}{H^{a}(H^{b}(p)) - p = 0}
    {
    f_{17} &=  {\LRL_1}{\LRL_2}+{\LLR_1}{\LRL_3}+{\LLL_1}{\LRL_4}+{\LRL_1}{\LRR_1}   \\
    f_{18} &={\LRL_2}{\LRL_2}+{\LLR_2}{\LRL_3}+{\LLL_2}{\LRL_4}+{\LRL_1}{\LRR_2}-1 \\
    f_{19} &={\LLR_3}{\LRL_3}+{\LRL_2}{\LRL_3}+{\LLL_3}{\LRL_4}+{\LRL_1}{\LRR_3}   \\
    f_{20} &={\LLR_4}{\LRL_3}+{\LLL_4}{\LRL_4}+{\LRL_2}{\LRL_4}+{\LRL_1}{\LRR_4}       }

    \ttab{21}{24}{p = a(bc)}{H^{a}(H^{c}(p)) - p = 0}
    {
    f_{21} &=  {\LRR_1}{\LRR_1}+{\LRL_1}{\LRR_2}+{\LLR_1}{\LRR_3}+{\LLL_1}{\LRR_4}-1 \\
    f_{22} &={\LRL_2}{\LRR_2}+{\LRR_1}{\LRR_2}+{\LLR_2}{\LRR_3}+{\LLL_2}{\LRR_4}   \\
    f_{23} &={\LRL_3}{\LRR_2}+{\LLR_3}{\LRR_3}+{\LRR_1}{\LRR_3}+{\LLL_3}{\LRR_4}   \\
    f_{24} &={\LRL_4}{\LRR_2}+{\LLR_4}{\LRR_3}+{\LLL_4}{\LRR_4}+{\LRR_1}{\LRR_4}
    }

    \ttab{25}{28}{p = a(bc)}{H^{b}(H^{c}(p)) - H^{b}(p) = 0}
    {
    f_{25} &=  {\LRL_1}{\LRR_1}+{\LRR_1}{\LRR_2}+{\LLL_1}{\LRR_3}+{\LLR_1}{\LRR_4}-{\LRL_2} \\
    f_{26} &={\LRL_2}{\LRR_1}+{\LRR_2}{\LRR_2}+{\LLL_2}{\LRR_3}+{\LLR_2}{\LRR_4}-{\LRL_1} \\
    f_{27} &={\LRL_3}{\LRR_1}+{\LLL_3}{\LRR_3}+{\LRR_2}{\LRR_3}+{\LLR_3}{\LRR_4}-{\LRL_4} \\
    f_{28} &={\LRL_4}{\LRR_1}+{\LLL_4}{\LRR_3}+{\LLR_4}{\LRR_4}+{\LRR_2}{\LRR_4}-{\LRL_3}     }

    \ttab{29}{32}{p = a(bc)}{H^{c}(H^{b}(p)) - H^{c}(p) = 0}
    {
    f_{29} &=  {\LRL_1}{\LRL_1}+{\LLL_1}{\LRL_3}+{\LLR_1}{\LRL_4}+{\LRL_2}{\LRR_1}-{\LRR_2} \\
    f_{30} &={\LRL_1}{\LRL_2}+{\LLL_2}{\LRL_3}+{\LLR_2}{\LRL_4}+{\LRL_2}{\LRR_2}-{\LRR_1} \\
    f_{31} &={\LLL_3}{\LRL_3}+{\LRL_1}{\LRL_3}+{\LLR_3}{\LRL_4}+{\LRL_2}{\LRR_3}-{\LRR_4} \\
    f_{32} &={\LLL_4}{\LRL_3}+{\LLR_4}{\LRL_4}+{\LRL_1}{\LRL_4}+{\LRL_2}{\LRR_4}-{\LRR_3}     }
\end{multicols}

\ttab{33}{40}{p = a(b(cd))}{H^{ca}(H^{ac}(H^{d}(H^{b}(p)))) - H^{db}(H^{bd}(H^{c}(p))) = 0}
{
f_{33} &=  {\LLR_3}{\LRL_1}{\LRL_1}+{\LLR_1}{\LLR_1}{\LRL_3}+{\LLL_1}{\LLR_3}{\LRL_3}+{\LLR_1}{\LRL_1}{\LRR_1}+{\LLR_1}{\LRL_4}{\LRR_1}+{\LRL_2}{\LRR_1}{\LRR_1}+{\LRL_1}{\LRL_2}{\LRR_3}+{\LLL_1}{\LRL_4}{\LRR_3}-{\LRL_3}{\LRL_3}-{\LRL_1}{\LRR_3} \\
f_{34} &={\LLR_3}{\LRL_1}{\LRL_2}+{\LLR_1}{\LLR_2}{\LRL_3}+{\LLL_2}{\LLR_3}{\LRL_3}+{\LLR_2}{\LRL_4}{\LRR_1}+{\LLR_1}{\LRL_1}{\LRR_2}+{\LRL_2}{\LRR_1}{\LRR_2}+{\LRL_2}{\LRL_2}{\LRR_3}+{\LLL_2}{\LRL_4}{\LRR_3}-{\LRL_3}{\LRL_4}-{\LRL_2}{\LRR_3} \\
f_{35} &={\LLL_3}{\LLR_3}{\LRL_3}+{\LLR_1}{\LLR_3}{\LRL_3}+{\LLR_3}{\LRL_1}{\LRL_3}+{\LLR_3}{\LRL_4}{\LRR_1}+{\LLR_1}{\LRL_1}{\LRR_3}+{\LRL_2}{\LRL_3}{\LRR_3}+{\LLL_3}{\LRL_4}{\LRR_3}+{\LRL_2}{\LRR_1}{\LRR_3}-{\LRL_1}{\LRL_3}-{\LRL_1}{\LRR_1} \\
f_{36} &={\LLL_4}{\LLR_3}{\LRL_3}+{\LLR_1}{\LLR_4}{\LRL_3}+{\LLR_3}{\LRL_1}{\LRL_4}+{\LLR_4}{\LRL_4}{\LRR_1}+{\LLL_4}{\LRL_4}{\LRR_3}+{\LRL_2}{\LRL_4}{\LRR_3}+{\LLR_1}{\LRL_1}{\LRR_4}+{\LRL_2}{\LRR_1}{\LRR_4}-{\LRL_1}{\LRL_4}-{\LRL_2}{\LRR_1} \\
f_{37} &=  {\LLR_4}{\LRL_1}{\LRL_1}+{\LLR_1}{\LLR_2}{\LRL_3}+{\LLL_1}{\LLR_4}{\LRL_3}+{\LLR_2}{\LRL_1}{\LRR_1}+{\LLR_1}{\LRL_4}{\LRR_2}+{\LRL_2}{\LRR_1}{\LRR_2}+{\LRL_1}{\LRL_2}{\LRR_4}+{\LLL_1}{\LRL_4}{\LRR_4}-{\LRL_3}{\LRL_4}-{\LRL_1}{\LRR_4} \\
f_{38} &={\LLR_4}{\LRL_1}{\LRL_2}+{\LLR_2}{\LLR_2}{\LRL_3}+{\LLL_2}{\LLR_4}{\LRL_3}+{\LLR_2}{\LRL_1}{\LRR_2}+{\LLR_2}{\LRL_4}{\LRR_2}+{\LRL_2}{\LRR_2}{\LRR_2}+{\LRL_2}{\LRL_2}{\LRR_4}+{\LLL_2}{\LRL_4}{\LRR_4}-{\LRL_4}{\LRL_4}-{\LRL_2}{\LRR_4} \\
f_{39} &={\LLR_2}{\LLR_3}{\LRL_3}+{\LLL_3}{\LLR_4}{\LRL_3}+{\LLR_4}{\LRL_1}{\LRL_3}+{\LLR_3}{\LRL_4}{\LRR_2}+{\LLR_2}{\LRL_1}{\LRR_3}+{\LRL_2}{\LRR_2}{\LRR_3}+{\LRL_2}{\LRL_3}{\LRR_4}+{\LLL_3}{\LRL_4}{\LRR_4}-{\LRL_2}{\LRL_3}-{\LRL_1}{\LRR_2} \\
f_{40} &={\LLL_4}{\LLR_4}{\LRL_3}+{\LLR_2}{\LLR_4}{\LRL_3}+{\LLR_4}{\LRL_1}{\LRL_4}+{\LLR_4}{\LRL_4}{\LRR_2}+{\LLR_2}{\LRL_1}{\LRR_4}+{\LLL_4}{\LRL_4}{\LRR_4}+{\LRL_2}{\LRL_4}{\LRR_4}+{\LRL_2}{\LRR_2}{\LRR_4}-{\LRL_2}{\LRL_4}-{\LRL_2}{\LRR_2} }

\ttab{41}{48}{p = a(b(cd))}{H^{da}(H^{ad}(H^{c}(H^{b}(p)))) - H^{cb}(H^{bc}(H^{d}(p))) = 0}
{
f_{41} &=  {\LLL_3}{\LRL_1}{\LRL_1}+{\LLL_1}{\LLL_3}{\LRL_3}+{\LLL_1}{\LLR_1}{\LRL_3}+{\LRL_1}{\LRL_2}{\LRL_3}+{\LLR_1}{\LRL_1}{\LRL_4}+{\LLL_1}{\LRL_3}{\LRL_4}+{\LLL_1}{\LRL_1}{\LRR_1}+{\LRL_1}{\LRL_2}{\LRR_1}-{\LRL_3}{\LRR_3}-{\LRR_1}{\LRR_3}  \\
f_{42} &={\LLL_3}{\LRL_1}{\LRL_2}+{\LLL_2}{\LLL_3}{\LRL_3}+{\LLL_1}{\LLR_2}{\LRL_3}+{\LRL_2}{\LRL_2}{\LRL_3}+{\LLR_2}{\LRL_1}{\LRL_4}+{\LLL_2}{\LRL_3}{\LRL_4}+{\LLL_1}{\LRL_1}{\LRR_2}+{\LRL_1}{\LRL_2}{\LRR_2}-{\LRR_2}{\LRR_3}-{\LRL_3}{\LRR_4}  \\
f_{43} &={\LLL_3}{\LLL_3}{\LRL_3}+{\LLL_1}{\LLR_3}{\LRL_3}+{\LLL_3}{\LRL_1}{\LRL_3}+{\LRL_2}{\LRL_3}{\LRL_3}+{\LLR_3}{\LRL_1}{\LRL_4}+{\LLL_3}{\LRL_3}{\LRL_4}+{\LLL_1}{\LRL_1}{\LRR_3}+{\LRL_1}{\LRL_2}{\LRR_3}-{\LRR_1}{\LRR_1}-{\LRL_1}{\LRR_3}  \\
f_{44} &={\LLL_3}{\LLL_4}{\LRL_3}+{\LLL_1}{\LLR_4}{\LRL_3}+{\LLL_3}{\LRL_1}{\LRL_4}+{\LLR_4}{\LRL_1}{\LRL_4}+{\LLL_4}{\LRL_3}{\LRL_4}+{\LRL_2}{\LRL_3}{\LRL_4}+{\LLL_1}{\LRL_1}{\LRR_4}+{\LRL_1}{\LRL_2}{\LRR_4}-{\LRR_1}{\LRR_2}-{\LRL_1}{\LRR_4}  \\
f_{45} &={\LLL_4}{\LRL_1}{\LRL_1}+{\LLL_1}{\LLL_4}{\LRL_3}+{\LLL_2}{\LLR_1}{\LRL_3}+{\LLR_1}{\LRL_2}{\LRL_4}+{\LRL_1}{\LRL_2}{\LRL_4}+{\LLL_1}{\LRL_4}{\LRL_4}+{\LLL_2}{\LRL_1}{\LRR_1}+{\LRL_2}{\LRL_2}{\LRR_1}-{\LRL_4}{\LRR_3}-{\LRR_1}{\LRR_4}  \\
f_{46} &={\LLL_4}{\LRL_1}{\LRL_2}+{\LLL_2}{\LLL_4}{\LRL_3}+{\LLL_2}{\LLR_2}{\LRL_3}+{\LLR_2}{\LRL_2}{\LRL_4}+{\LRL_2}{\LRL_2}{\LRL_4}+{\LLL_2}{\LRL_4}{\LRL_4}+{\LLL_2}{\LRL_1}{\LRR_2}+{\LRL_2}{\LRL_2}{\LRR_2}-{\LRL_4}{\LRR_4}-{\LRR_2}{\LRR_4}  \\
f_{47} &={\LLL_3}{\LLL_4}{\LRL_3}+{\LLL_2}{\LLR_3}{\LRL_3}+{\LLL_4}{\LRL_1}{\LRL_3}+{\LLR_3}{\LRL_2}{\LRL_4}+{\LRL_2}{\LRL_3}{\LRL_4}+{\LLL_3}{\LRL_4}{\LRL_4}+{\LLL_2}{\LRL_1}{\LRR_3}+{\LRL_2}{\LRL_2}{\LRR_3}-{\LRR_1}{\LRR_2}-{\LRL_2}{\LRR_3}  \\
f_{48} &={\LLL_4}{\LLL_4}{\LRL_3}+{\LLL_2}{\LLR_4}{\LRL_3}+{\LLL_4}{\LRL_1}{\LRL_4}+{\LLR_4}{\LRL_2}{\LRL_4}+{\LLL_4}{\LRL_4}{\LRL_4}+{\LRL_2}{\LRL_4}{\LRL_4}+{\LLL_2}{\LRL_1}{\LRR_4}+{\LRL_2}{\LRL_2}{\LRR_4}-{\LRR_2}{\LRR_2}-{\LRL_2}{\LRR_4}
}

\ttab{49}{56}{p = a((bc)d)}{H^{ca}(H^{ac}(H^{b}(H^{d}(p)))) - H^{db}(H^{bd}(H^{c}(p))) = 0}
{
f_{49} &=  {\LLL_3}{\LRL_1}{\LRR_1}+{\LLL_1}{\LRR_1}{\LRR_1}+{\LRL_1}{\LRL_3}{\LRR_2}+{\LRL_1}{\LRR_1}{\LRR_2}+{\LLL_1}{\LLL_3}{\LRR_3}+{\LLL_1}{\LLR_1}{\LRR_3}+{\LLR_1}{\LRL_1}{\LRR_4}+{\LLL_1}{\LRL_3}{\LRR_4}-{\LLR_3}{\LRL_3}-{\LLR_1}{\LRR_3} \\
f_{50} &={\LLL_3}{\LRL_2}{\LRR_1}+{\LRL_2}{\LRL_3}{\LRR_2}+{\LLL_1}{\LRR_1}{\LRR_2}+{\LRL_1}{\LRR_2}{\LRR_2}+{\LLL_2}{\LLL_3}{\LRR_3}+{\LLL_1}{\LLR_2}{\LRR_3}+{\LLR_2}{\LRL_1}{\LRR_4}+{\LLL_2}{\LRL_3}{\LRR_4}-{\LLR_4}{\LRL_3}-{\LLR_2}{\LRR_3} \\
f_{51} &={\LLL_3}{\LRL_3}{\LRR_1}+{\LRL_3}{\LRL_3}{\LRR_2}+{\LLL_3}{\LLL_3}{\LRR_3}+{\LLL_1}{\LLR_3}{\LRR_3}+{\LLL_1}{\LRR_1}{\LRR_3}+{\LRL_1}{\LRR_2}{\LRR_3}+{\LLR_3}{\LRL_1}{\LRR_4}+{\LLL_3}{\LRL_3}{\LRR_4}-{\LLR_3}{\LRL_1}-{\LLR_1}{\LRR_1} \\
f_{52} &={\LLL_3}{\LRL_4}{\LRR_1}+{\LRL_3}{\LRL_4}{\LRR_2}+{\LLL_3}{\LLL_4}{\LRR_3}+{\LLL_1}{\LLR_4}{\LRR_3}+{\LLR_4}{\LRL_1}{\LRR_4}+{\LLL_4}{\LRL_3}{\LRR_4}+{\LLL_1}{\LRR_1}{\LRR_4}+{\LRL_1}{\LRR_2}{\LRR_4}-{\LLR_4}{\LRL_1}-{\LLR_2}{\LRR_1} \\
f_{53} &={\LLL_4}{\LRL_1}{\LRR_1}+{\LLL_2}{\LRR_1}{\LRR_1}+{\LRL_1}{\LRL_4}{\LRR_2}+{\LRL_2}{\LRR_1}{\LRR_2}+{\LLL_1}{\LLL_4}{\LRR_3}+{\LLL_2}{\LLR_1}{\LRR_3}+{\LLR_1}{\LRL_2}{\LRR_4}+{\LLL_1}{\LRL_4}{\LRR_4}-{\LLR_3}{\LRL_4}-{\LLR_1}{\LRR_4} \\
f_{54} &={\LLL_4}{\LRL_2}{\LRR_1}+{\LRL_2}{\LRL_4}{\LRR_2}+{\LLL_2}{\LRR_1}{\LRR_2}+{\LRL_2}{\LRR_2}{\LRR_2}+{\LLL_2}{\LLL_4}{\LRR_3}+{\LLL_2}{\LLR_2}{\LRR_3}+{\LLR_2}{\LRL_2}{\LRR_4}+{\LLL_2}{\LRL_4}{\LRR_4}-{\LLR_4}{\LRL_4}-{\LLR_2}{\LRR_4} \\
f_{55} &={\LLL_4}{\LRL_3}{\LRR_1}+{\LRL_3}{\LRL_4}{\LRR_2}+{\LLL_3}{\LLL_4}{\LRR_3}+{\LLL_2}{\LLR_3}{\LRR_3}+{\LLL_2}{\LRR_1}{\LRR_3}+{\LRL_2}{\LRR_2}{\LRR_3}+{\LLR_3}{\LRL_2}{\LRR_4}+{\LLL_3}{\LRL_4}{\LRR_4}-{\LLR_3}{\LRL_2}-{\LLR_1}{\LRR_2} \\
f_{56} &={\LLL_4}{\LRL_4}{\LRR_1}+{\LRL_4}{\LRL_4}{\LRR_2}+{\LLL_4}{\LLL_4}{\LRR_3}+{\LLL_2}{\LLR_4}{\LRR_3}+{\LLR_4}{\LRL_2}{\LRR_4}+{\LLL_4}{\LRL_4}{\LRR_4}+{\LLL_2}{\LRR_1}{\LRR_4}+{\LRL_2}{\LRR_2}{\LRR_4}-{\LLR_4}{\LRL_2}-{\LLR_2}{\LRR_2} }

\ttab{57}{64}{p = a((bc)d)}{H^{ba}(H^{ab}(H^{c}(H^{d}(p)))) - H^{dc}(H^{cd}(H^{b}(p))) = 0}
{
f_{57} &=  {\LLR_3}{\LRL_1}{\LRR_1}+{\LLR_1}{\LRR_1}{\LRR_1}+{\LRR_1}{\LRR_1}{\LRR_2}+{\LLR_1}{\LLR_1}{\LRR_3}+{\LLL_1}{\LLR_3}{\LRR_3}+{\LRL_1}{\LRR_2}{\LRR_3}+{\LLR_1}{\LRR_1}{\LRR_4}+{\LLL_1}{\LRR_3}{\LRR_4}-{\LLL_3}{\LRL_3}-{\LLL_1}{\LRR_3} \\
f_{58} &={\LLR_3}{\LRL_2}{\LRR_1}+{\LLR_1}{\LRR_1}{\LRR_2}+{\LRR_1}{\LRR_2}{\LRR_2}+{\LLR_1}{\LLR_2}{\LRR_3}+{\LLL_2}{\LLR_3}{\LRR_3}+{\LRL_2}{\LRR_2}{\LRR_3}+{\LLR_2}{\LRR_1}{\LRR_4}+{\LLL_2}{\LRR_3}{\LRR_4}-{\LLL_4}{\LRL_3}-{\LLL_2}{\LRR_3} \\
f_{59} &={\LLR_3}{\LRL_3}{\LRR_1}+{\LLL_3}{\LLR_3}{\LRR_3}+{\LLR_1}{\LLR_3}{\LRR_3}+{\LLR_1}{\LRR_1}{\LRR_3}+{\LRL_3}{\LRR_2}{\LRR_3}+{\LRR_1}{\LRR_2}{\LRR_3}+{\LLR_3}{\LRR_1}{\LRR_4}+{\LLL_3}{\LRR_3}{\LRR_4}-{\LLL_3}{\LRL_1}-{\LLL_1}{\LRR_1} \\
f_{60} &={\LLR_3}{\LRL_4}{\LRR_1}+{\LLL_4}{\LLR_3}{\LRR_3}+{\LLR_1}{\LLR_4}{\LRR_3}+{\LRL_4}{\LRR_2}{\LRR_3}+{\LLR_1}{\LRR_1}{\LRR_4}+{\LLR_4}{\LRR_1}{\LRR_4}+{\LRR_1}{\LRR_2}{\LRR_4}+{\LLL_4}{\LRR_3}{\LRR_4}-{\LLL_4}{\LRL_1}-{\LLL_2}{\LRR_1} \\
f_{61} &={\LLR_4}{\LRL_1}{\LRR_1}+{\LLR_2}{\LRR_1}{\LRR_1}+{\LRR_1}{\LRR_2}{\LRR_2}+{\LLR_1}{\LLR_2}{\LRR_3}+{\LLL_1}{\LLR_4}{\LRR_3}+{\LLR_1}{\LRR_2}{\LRR_4}+{\LRL_1}{\LRR_2}{\LRR_4}+{\LLL_1}{\LRR_4}{\LRR_4}-{\LLL_3}{\LRL_4}-{\LLL_1}{\LRR_4} \\
f_{62} &={\LLR_4}{\LRL_2}{\LRR_1}+{\LLR_2}{\LRR_1}{\LRR_2}+{\LRR_2}{\LRR_2}{\LRR_2}+{\LLR_2}{\LLR_2}{\LRR_3}+{\LLL_2}{\LLR_4}{\LRR_3}+{\LLR_2}{\LRR_2}{\LRR_4}+{\LRL_2}{\LRR_2}{\LRR_4}+{\LLL_2}{\LRR_4}{\LRR_4}-{\LLL_4}{\LRL_4}-{\LLL_2}{\LRR_4} \\
f_{63} &={\LLR_4}{\LRL_3}{\LRR_1}+{\LLR_2}{\LLR_3}{\LRR_3}+{\LLL_3}{\LLR_4}{\LRR_3}+{\LLR_2}{\LRR_1}{\LRR_3}+{\LRR_2}{\LRR_2}{\LRR_3}+{\LLR_3}{\LRR_2}{\LRR_4}+{\LRL_3}{\LRR_2}{\LRR_4}+{\LLL_3}{\LRR_4}{\LRR_4}-{\LLL_3}{\LRL_2}-{\LLL_1}{\LRR_2} \\
f_{64} &={\LLR_4}{\LRL_4}{\LRR_1}+{\LLL_4}{\LLR_4}{\LRR_3}+{\LLR_2}{\LLR_4}{\LRR_3}+{\LLR_2}{\LRR_1}{\LRR_4}+{\LLR_4}{\LRR_2}{\LRR_4}+{\LRL_4}{\LRR_2}{\LRR_4}+{\LRR_2}{\LRR_2}{\LRR_4}+{\LLL_4}{\LRR_4}{\LRR_4}-{\LLL_4}{\LRL_2}-{\LLL_2}{\LRR_2} }

\ttab{65}{72}{p = (a(bc))d}{H^{cd}(H^{dc}(H^{b}(H^{a}(p)))) - H^{ab}(H^{ba}(H^{c}(p))) = 0}
{
f_{65} &=  {\LLL_1}{\LLL_3}{\LLL_3}+{\LLL_1}{\LLL_3}{\LLR_1}+{\LLL_1}{\LLL_3}{\LRL_1}+{\LLL_4}{\LLR_1}{\LRL_1}+{\LLL_1}{\LLL_4}{\LRL_3}+{\LLL_2}{\LRL_1}{\LRL_3}+{\LLL_1}{\LLL_1}{\LRR_1}+{\LLL_2}{\LRL_1}{\LRR_1}-{\LLR_3}{\LRR_1}-{\LLL_3}{\LRR_3} \\
f_{66} &={\LLL_2}{\LLL_3}{\LLL_3}+{\LLL_1}{\LLL_3}{\LLR_2}+{\LLL_4}{\LLR_2}{\LRL_1}+{\LLL_1}{\LLL_3}{\LRL_2}+{\LLL_2}{\LLL_4}{\LRL_3}+{\LLL_2}{\LRL_2}{\LRL_3}+{\LLL_1}{\LLL_1}{\LRR_2}+{\LLL_2}{\LRL_1}{\LRR_2}-{\LLR_3}{\LRR_2}-{\LLL_3}{\LRR_4} \\
f_{67} &={\LLL_3}{\LLL_3}{\LLL_3}+{\LLL_1}{\LLL_3}{\LLR_3}+{\LLL_4}{\LLR_3}{\LRL_1}+{\LLL_1}{\LLL_3}{\LRL_3}+{\LLL_3}{\LLL_4}{\LRL_3}+{\LLL_2}{\LRL_3}{\LRL_3}+{\LLL_1}{\LLL_1}{\LRR_3}+{\LLL_2}{\LRL_1}{\LRR_3}-{\LLR_1}{\LRR_1}-{\LLL_1}{\LRR_3} \\
f_{68} &={\LLL_3}{\LLL_3}{\LLL_4}+{\LLL_1}{\LLL_3}{\LLR_4}+{\LLL_4}{\LLR_4}{\LRL_1}+{\LLL_4}{\LLL_4}{\LRL_3}+{\LLL_1}{\LLL_3}{\LRL_4}+{\LLL_2}{\LRL_3}{\LRL_4}+{\LLL_1}{\LLL_1}{\LRR_4}+{\LLL_2}{\LRL_1}{\LRR_4}-{\LLR_1}{\LRR_2}-{\LLL_1}{\LRR_4} \\
f_{69} &={\LLL_1}{\LLL_3}{\LLL_4}+{\LLL_2}{\LLL_3}{\LLR_1}+{\LLL_1}{\LLL_4}{\LRL_1}+{\LLL_4}{\LLR_1}{\LRL_2}+{\LLL_1}{\LLL_4}{\LRL_4}+{\LLL_2}{\LRL_1}{\LRL_4}+{\LLL_1}{\LLL_2}{\LRR_1}+{\LLL_2}{\LRL_2}{\LRR_1}-{\LLR_4}{\LRR_1}-{\LLL_4}{\LRR_3} \\
f_{70} &={\LLL_2}{\LLL_3}{\LLL_4}+{\LLL_2}{\LLL_3}{\LLR_2}+{\LLL_1}{\LLL_4}{\LRL_2}+{\LLL_4}{\LLR_2}{\LRL_2}+{\LLL_2}{\LLL_4}{\LRL_4}+{\LLL_2}{\LRL_2}{\LRL_4}+{\LLL_1}{\LLL_2}{\LRR_2}+{\LLL_2}{\LRL_2}{\LRR_2}-{\LLR_4}{\LRR_2}-{\LLL_4}{\LRR_4} \\
f_{71} &={\LLL_3}{\LLL_3}{\LLL_4}+{\LLL_2}{\LLL_3}{\LLR_3}+{\LLL_4}{\LLR_3}{\LRL_2}+{\LLL_1}{\LLL_4}{\LRL_3}+{\LLL_3}{\LLL_4}{\LRL_4}+{\LLL_2}{\LRL_3}{\LRL_4}+{\LLL_1}{\LLL_2}{\LRR_3}+{\LLL_2}{\LRL_2}{\LRR_3}-{\LLR_2}{\LRR_1}-{\LLL_2}{\LRR_3} \\
f_{72} &={\LLL_3}{\LLL_4}{\LLL_4}+{\LLL_2}{\LLL_3}{\LLR_4}+{\LLL_4}{\LLR_4}{\LRL_2}+{\LLL_1}{\LLL_4}{\LRL_4}+{\LLL_4}{\LLL_4}{\LRL_4}+{\LLL_2}{\LRL_4}{\LRL_4}+{\LLL_1}{\LLL_2}{\LRR_4}+{\LLL_2}{\LRL_2}{\LRR_4}-{\LLR_2}{\LRR_2}-{\LLL_2}{\LRR_4} }

\ttab{73}{80}{p = (a(bc))d}{H^{bd}(H^{db}(H^{c}(H^{a}(p)))) - H^{ac}(H^{ca}(H^{b}(p))) = 0}
{
f_{73} &=  {\LLL_3}{\LLR_1}{\LLR_1}+{\LLL_1}{\LLL_3}{\LLR_3}+{\LLL_1}{\LLR_3}{\LRL_1}+{\LLL_1}{\LLR_1}{\LRR_1}+{\LLL_4}{\LLR_1}{\LRR_1}+{\LLL_2}{\LRR_1}{\LRR_1}+{\LLL_1}{\LLL_4}{\LRR_3}+{\LLL_2}{\LRL_1}{\LRR_3}-{\LLR_3}{\LRL_1}-{\LLL_3}{\LRL_3} \\
f_{74} &={\LLL_3}{\LLR_1}{\LLR_2}+{\LLL_2}{\LLL_3}{\LLR_3}+{\LLL_1}{\LLR_3}{\LRL_2}+{\LLL_4}{\LLR_2}{\LRR_1}+{\LLL_1}{\LLR_1}{\LRR_2}+{\LLL_2}{\LRR_1}{\LRR_2}+{\LLL_2}{\LLL_4}{\LRR_3}+{\LLL_2}{\LRL_2}{\LRR_3}-{\LLR_3}{\LRL_2}-{\LLL_3}{\LRL_4} \\
f_{75} &={\LLL_3}{\LLL_3}{\LLR_3}+{\LLL_3}{\LLR_1}{\LLR_3}+{\LLL_1}{\LLR_3}{\LRL_3}+{\LLL_4}{\LLR_3}{\LRR_1}+{\LLL_3}{\LLL_4}{\LRR_3}+{\LLL_1}{\LLR_1}{\LRR_3}+{\LLL_2}{\LRL_3}{\LRR_3}+{\LLL_2}{\LRR_1}{\LRR_3}-{\LLR_1}{\LRL_1}-{\LLL_1}{\LRL_3} \\
f_{76} &={\LLL_3}{\LLL_4}{\LLR_3}+{\LLL_3}{\LLR_1}{\LLR_4}+{\LLL_1}{\LLR_3}{\LRL_4}+{\LLL_4}{\LLR_4}{\LRR_1}+{\LLL_4}{\LLL_4}{\LRR_3}+{\LLL_2}{\LRL_4}{\LRR_3}+{\LLL_1}{\LLR_1}{\LRR_4}+{\LLL_2}{\LRR_1}{\LRR_4}-{\LLR_1}{\LRL_2}-{\LLL_1}{\LRL_4} \\
f_{77} &={\LLL_3}{\LLR_1}{\LLR_2}+{\LLL_1}{\LLL_3}{\LLR_4}+{\LLL_1}{\LLR_4}{\LRL_1}+{\LLL_1}{\LLR_2}{\LRR_1}+{\LLL_4}{\LLR_1}{\LRR_2}+{\LLL_2}{\LRR_1}{\LRR_2}+{\LLL_1}{\LLL_4}{\LRR_4}+{\LLL_2}{\LRL_1}{\LRR_4}-{\LLR_4}{\LRL_1}-{\LLL_4}{\LRL_3} \\
f_{78} &={\LLL_3}{\LLR_2}{\LLR_2}+{\LLL_2}{\LLL_3}{\LLR_4}+{\LLL_1}{\LLR_4}{\LRL_2}+{\LLL_1}{\LLR_2}{\LRR_2}+{\LLL_4}{\LLR_2}{\LRR_2}+{\LLL_2}{\LRR_2}{\LRR_2}+{\LLL_2}{\LLL_4}{\LRR_4}+{\LLL_2}{\LRL_2}{\LRR_4}-{\LLR_4}{\LRL_2}-{\LLL_4}{\LRL_4} \\
f_{79} &={\LLL_3}{\LLR_2}{\LLR_3}+{\LLL_3}{\LLL_3}{\LLR_4}+{\LLL_1}{\LLR_4}{\LRL_3}+{\LLL_4}{\LLR_3}{\LRR_2}+{\LLL_1}{\LLR_2}{\LRR_3}+{\LLL_2}{\LRR_2}{\LRR_3}+{\LLL_3}{\LLL_4}{\LRR_4}+{\LLL_2}{\LRL_3}{\LRR_4}-{\LLR_2}{\LRL_1}-{\LLL_2}{\LRL_3} \\
f_{80} &={\LLL_3}{\LLL_4}{\LLR_4}+{\LLL_3}{\LLR_2}{\LLR_4}+{\LLL_1}{\LLR_4}{\LRL_4}+{\LLL_4}{\LLR_4}{\LRR_2}+{\LLL_4}{\LLL_4}{\LRR_4}+{\LLL_1}{\LLR_2}{\LRR_4}+{\LLL_2}{\LRL_4}{\LRR_4}+{\LLL_2}{\LRR_2}{\LRR_4}-{\LLR_2}{\LRL_2}-{\LLL_2}{\LRL_4} }

\ttab{81}{88}{p = ((ab)c)d}{H^{bd}(H^{db}(H^{a}(H^{c}(p)))) - H^{ac}(H^{ca}(H^{b}(p))) = 0}
{
f_{81} &=  {\LLL_1}{\LLL_3}{\LLR_3}+{\LLL_1}{\LLR_1}{\LLR_3}+{\LLL_3}{\LLR_1}{\LRL_1}+{\LLR_1}{\LLR_4}{\LRL_1}+{\LLL_1}{\LLR_4}{\LRL_3}+{\LLR_2}{\LRL_1}{\LRL_3}+{\LLL_1}{\LLR_1}{\LRR_1}+{\LLR_2}{\LRL_1}{\LRR_1}-{\LLL_3}{\LLR_3}-{\LLR_1}{\LLR_3} \\
f_{82} &={\LLL_2}{\LLL_3}{\LLR_3}+{\LLL_1}{\LLR_2}{\LLR_3}+{\LLR_2}{\LLR_4}{\LRL_1}+{\LLL_3}{\LLR_1}{\LRL_2}+{\LLL_2}{\LLR_4}{\LRL_3}+{\LLR_2}{\LRL_2}{\LRL_3}+{\LLL_1}{\LLR_1}{\LRR_2}+{\LLR_2}{\LRL_1}{\LRR_2}-{\LLR_2}{\LLR_3}-{\LLL_3}{\LLR_4} \\
f_{83} &={\LLL_3}{\LLL_3}{\LLR_3}+{\LLL_1}{\LLR_3}{\LLR_3}+{\LLR_3}{\LLR_4}{\LRL_1}+{\LLL_3}{\LLR_1}{\LRL_3}+{\LLL_3}{\LLR_4}{\LRL_3}+{\LLR_2}{\LRL_3}{\LRL_3}+{\LLL_1}{\LLR_1}{\LRR_3}+{\LLR_2}{\LRL_1}{\LRR_3}-{\LLR_1}{\LLR_1}-{\LLL_1}{\LLR_3} \\
f_{84} &={\LLL_3}{\LLL_4}{\LLR_3}+{\LLL_1}{\LLR_3}{\LLR_4}+{\LLR_4}{\LLR_4}{\LRL_1}+{\LLL_4}{\LLR_4}{\LRL_3}+{\LLL_3}{\LLR_1}{\LRL_4}+{\LLR_2}{\LRL_3}{\LRL_4}+{\LLL_1}{\LLR_1}{\LRR_4}+{\LLR_2}{\LRL_1}{\LRR_4}-{\LLR_1}{\LLR_2}-{\LLL_1}{\LLR_4} \\
f_{85} &={\LLL_1}{\LLL_4}{\LLR_3}+{\LLL_2}{\LLR_1}{\LLR_3}+{\LLL_4}{\LLR_1}{\LRL_1}+{\LLR_1}{\LLR_4}{\LRL_2}+{\LLL_1}{\LLR_4}{\LRL_4}+{\LLR_2}{\LRL_1}{\LRL_4}+{\LLL_2}{\LLR_1}{\LRR_1}+{\LLR_2}{\LRL_2}{\LRR_1}-{\LLL_4}{\LLR_3}-{\LLR_1}{\LLR_4} \\
f_{86} &={\LLL_2}{\LLL_4}{\LLR_3}+{\LLL_2}{\LLR_2}{\LLR_3}+{\LLL_4}{\LLR_1}{\LRL_2}+{\LLR_2}{\LLR_4}{\LRL_2}+{\LLL_2}{\LLR_4}{\LRL_4}+{\LLR_2}{\LRL_2}{\LRL_4}+{\LLL_2}{\LLR_1}{\LRR_2}+{\LLR_2}{\LRL_2}{\LRR_2}-{\LLL_4}{\LLR_4}-{\LLR_2}{\LLR_4} \\
f_{87} &={\LLL_3}{\LLL_4}{\LLR_3}+{\LLL_2}{\LLR_3}{\LLR_3}+{\LLR_3}{\LLR_4}{\LRL_2}+{\LLL_4}{\LLR_1}{\LRL_3}+{\LLL_3}{\LLR_4}{\LRL_4}+{\LLR_2}{\LRL_3}{\LRL_4}+{\LLL_2}{\LLR_1}{\LRR_3}+{\LLR_2}{\LRL_2}{\LRR_3}-{\LLR_1}{\LLR_2}-{\LLL_2}{\LLR_3} \\
f_{88} &={\LLL_4}{\LLL_4}{\LLR_3}+{\LLL_2}{\LLR_3}{\LLR_4}+{\LLR_4}{\LLR_4}{\LRL_2}+{\LLL_4}{\LLR_1}{\LRL_4}+{\LLL_4}{\LLR_4}{\LRL_4}+{\LLR_2}{\LRL_4}{\LRL_4}+{\LLL_2}{\LLR_1}{\LRR_4}+{\LLR_2}{\LRL_2}{\LRR_4}-{\LLR_2}{\LLR_2}-{\LLL_2}{\LLR_4} }

\ttab{89}{96}{p = ((ab)c)d}{H^{ad}(H^{da}(H^{b}(H^{c}(p)))) - H^{bc}(H^{cb}(H^{a}(p))) = 0}
{
f_{89} &=  {\LLR_1}{\LLR_1}{\LLR_3}+{\LLL_1}{\LLR_3}{\LLR_3}+{\LLR_1}{\LLR_3}{\LRL_1}+{\LLR_1}{\LLR_1}{\LRR_1}+{\LLR_1}{\LLR_4}{\LRR_1}+{\LLR_2}{\LRR_1}{\LRR_1}+{\LLL_1}{\LLR_4}{\LRR_3}+{\LLR_2}{\LRL_1}{\LRR_3}-{\LLL_3}{\LLL_3}-{\LLL_1}{\LLR_3} \\
f_{90} &={\LLR_1}{\LLR_2}{\LLR_3}+{\LLL_2}{\LLR_3}{\LLR_3}+{\LLR_1}{\LLR_3}{\LRL_2}+{\LLR_2}{\LLR_4}{\LRR_1}+{\LLR_1}{\LLR_1}{\LRR_2}+{\LLR_2}{\LRR_1}{\LRR_2}+{\LLL_2}{\LLR_4}{\LRR_3}+{\LLR_2}{\LRL_2}{\LRR_3}-{\LLL_3}{\LLL_4}-{\LLL_2}{\LLR_3} \\
f_{91} &={\LLL_3}{\LLR_3}{\LLR_3}+{\LLR_1}{\LLR_3}{\LLR_3}+{\LLR_1}{\LLR_3}{\LRL_3}+{\LLR_3}{\LLR_4}{\LRR_1}+{\LLR_1}{\LLR_1}{\LRR_3}+{\LLL_3}{\LLR_4}{\LRR_3}+{\LLR_2}{\LRL_3}{\LRR_3}+{\LLR_2}{\LRR_1}{\LRR_3}-{\LLL_1}{\LLL_3}-{\LLL_1}{\LLR_1} \\
f_{92} &={\LLL_4}{\LLR_3}{\LLR_3}+{\LLR_1}{\LLR_3}{\LLR_4}+{\LLR_1}{\LLR_3}{\LRL_4}+{\LLR_4}{\LLR_4}{\LRR_1}+{\LLL_4}{\LLR_4}{\LRR_3}+{\LLR_2}{\LRL_4}{\LRR_3}+{\LLR_1}{\LLR_1}{\LRR_4}+{\LLR_2}{\LRR_1}{\LRR_4}-{\LLL_1}{\LLL_4}-{\LLL_2}{\LLR_1} \\
f_{93} &={\LLR_1}{\LLR_2}{\LLR_3}+{\LLL_1}{\LLR_3}{\LLR_4}+{\LLR_1}{\LLR_4}{\LRL_1}+{\LLR_1}{\LLR_2}{\LRR_1}+{\LLR_1}{\LLR_4}{\LRR_2}+{\LLR_2}{\LRR_1}{\LRR_2}+{\LLL_1}{\LLR_4}{\LRR_4}+{\LLR_2}{\LRL_1}{\LRR_4}-{\LLL_3}{\LLL_4}-{\LLL_1}{\LLR_4} \\
f_{94} &={\LLR_2}{\LLR_2}{\LLR_3}+{\LLL_2}{\LLR_3}{\LLR_4}+{\LLR_1}{\LLR_4}{\LRL_2}+{\LLR_1}{\LLR_2}{\LRR_2}+{\LLR_2}{\LLR_4}{\LRR_2}+{\LLR_2}{\LRR_2}{\LRR_2}+{\LLL_2}{\LLR_4}{\LRR_4}+{\LLR_2}{\LRL_2}{\LRR_4}-{\LLL_4}{\LLL_4}-{\LLL_2}{\LLR_4} \\
f_{95} &={\LLR_2}{\LLR_3}{\LLR_3}+{\LLL_3}{\LLR_3}{\LLR_4}+{\LLR_1}{\LLR_4}{\LRL_3}+{\LLR_3}{\LLR_4}{\LRR_2}+{\LLR_1}{\LLR_2}{\LRR_3}+{\LLR_2}{\LRR_2}{\LRR_3}+{\LLL_3}{\LLR_4}{\LRR_4}+{\LLR_2}{\LRL_3}{\LRR_4}-{\LLL_2}{\LLL_3}-{\LLL_1}{\LLR_2} \\
f_{96} &={\LLL_4}{\LLR_3}{\LLR_4}+{\LLR_2}{\LLR_3}{\LLR_4}+{\LLR_1}{\LLR_4}{\LRL_4}+{\LLR_4}{\LLR_4}{\LRR_2}+{\LLR_1}{\LLR_2}{\LRR_4}+{\LLL_4}{\LLR_4}{\LRR_4}+{\LLR_2}{\LRL_4}{\LRR_4}+{\LLR_2}{\LRR_2}{\LRR_4}-{\LLL_2}{\LLL_4}-{\LLL_2}{\LLR_2} }

\normalsize

\end{document}